\documentclass[11pt,reqno]{amsart}

\usepackage{enumerate}
\usepackage{amsmath, amssymb, amsthm}
\usepackage{mathrsfs}
\usepackage{esint}
\usepackage{xcolor}
\usepackage{mathtools}
\usepackage{hyperref}
\usepackage{bm}
\usepackage{kotex}

\numberwithin{equation}{section}

\newtheorem{theorem}{Theorem}[section]
\newtheorem{lemma}[theorem]{Lemma}
\newtheorem{corollary}[theorem]{Corollary}
\newtheorem{remark}[theorem]{\bf{Remark}}
\newtheorem{assumption}[theorem]{Assumption}

\newtheorem{definition}[theorem]{Definition}

\theoremstyle{remark}
\theoremstyle{definition}

\newcommand\bL{\mathbb{L}}
\newcommand\bR{\mathbb{R}}

\newcommand\bH{\mathbb{H}}

\newcommand\bZ{\mathbb{Z}}

\newcommand\bE{\mathbb{E}}
\newcommand\bN{\mathbb{N}}

\newcommand\cB{\mathcal{B}}
\newcommand\cC{\mathcal{C}}

\newcommand\cF{\mathcal{F}}

\newcommand\cH{\mathcal{H}}
\newcommand\cI{\mathcal{I}}

\newcommand\cL{\mathcal{L}}
\newcommand\cP{\mathcal{P}}

\newcommand\cS{\mathcal{S}}
\newcommand\cM{\mathcal{M}}

\newcommand\cbrk{\text{$]$\kern-.15em$]$}}
\newcommand\opar{\text{\,\raise.2ex\hbox{${\scriptstyle
|}$}\kern-.34em$($}}
\newcommand\cpar{\text{$)$\kern-.34em\raise.2ex\hbox{${\scriptstyle |}$}}\,}

\newcommand\ep{\varepsilon}

\begin{document}

\title[A 
regularity theory for SGBE driven by a space-time white noise]{A 
regularity theory for stochastic generalized Burgers' equation driven by a multiplicative space-time white noise
}

\author{Beom-Seok Han 
}

\address{Department of Mathematics, Pohang University of Science and Technology, 77, Cheongam-ro, Nam-gu, Pohang, Gyeongbuk, 37673, Republic of Korea}

\email{hanbeom@postech.ac.kr}

\thanks{This work was supported by the National Research Foundation of Korea (NRF) grant
funded by the Korea government (MSIT) (No. NRF-2021R1C1C2007792) 
and the BK21 FOUR (Fostering Outstanding Universities for Research) funded by the Ministry of Education (MOE, Korea) and National Research Foundation of Korea (NRF)}

\subjclass[2020]{60H15, 35R60}

\keywords{Stochastic partial differential equation,
Nonlinear,
Super-linear,
Stochastic generalized Burger's equation,
Space-time white noise,
H\"older regularity}

\begin{abstract}

We introduce the uniqueness, existence, $L_p$-regularity, and maximal H\"older regularity of the solution to semilinear stochastic partial differential equation driven by a multiplicative space-time white noise:
$$ u_t = au_{xx} + bu_{x} + cu + \bar b|u|^\lambda u_{x} + \sigma(u)\dot W,\quad (t,x)\in(0,\infty)\times\bR; \quad u(0,\cdot) = u_0,
$$
where $\lambda > 0$. The function $\sigma(u)$ is either bounded Lipschitz or super-linear in $u$. The noise $\dot W$ is a space-time white noise. The coefficients $a,b,c$ depend on $(\omega,t,x)$, and $\bar b$ depends on $(\omega,t)$. The coefficients $a,b,c,\bar{b}$ are uniformly bounded, and $a$ satisfies ellipticity condition. The random initial data $u_0 = u_0(\omega,x)$ is nonnegative.

To establish the $L_p$-regularity theory, we impose an algebraic condition on $\lambda$ depending on the nonlinearity of the diffusion coefficient $\sigma(u)$. For example, if $\sigma(u)$ has Lipschitz continuity, linear growth, and \textit{boundedness} in $u$, $\lambda$ is assumed to be less than or \textit{equal} to $1$; $\lambda\in (0,1]$. However, if $\sigma(u) = |u|^{1+\lambda_0}$ with $\lambda_0\in[0,1/2)$, $\lambda$ is taken to be less than $1$; $\lambda\in(0,1)$. Under those conditions, the uniqueness, existence, and regularity of the solution are obtained in stochastic $L_p$ spaces. Also, we have the maximal H\"older regularity by employing the H\"older embedding theorem. For example, if $\lambda \in(0,1]$ and $\sigma(u)$ has Lipschitz continuity, linear growth, and boundedness in $u$, for $T<\infty$ and $\ep>0$,  
\begin{equation*}
u \in C^{1/4 - \ep,1/2 - \ep}_{t,x}([0,T]\times\bR)\quad(a.s.).
\end{equation*}
On the other hand, if $\lambda\in(0,1)$ and $\sigma(u) = |u|^{1+\lambda_0}$ with $\lambda_0\in[0,1/2)$, for $T<\infty$ and $\ep>0$,
\begin{equation*}
u \in C^{\frac{1/2-(\lambda -1/2) \vee \lambda_0}{2} - \ep,1/2-(\lambda -1/2) \vee \lambda_0 - \ep}_{t,x}([0,T]\times\bR)\quad (a.s.).
\end{equation*}
It should be noted that if $\sigma(u)$ is bounded Lipschitz in $u$, the H\"older regularity of the solution is independent of $\lambda$. However, if $\sigma(u)$ is super-linear in $u$, the H\"older regularities of the solution are affected by nonlinearities, $\lambda$ and $\lambda_0.$

\end{abstract}

\maketitle

\section{Introduction}

One of the most well-known example of semilinear stochastic partial differential equation (SPDE) is a stochastic generalized Burgers' equation driven by space-time white noise $\dot W$:
\begin{equation}
\label{burgers_type_equation}
u_t =  Lu + f(u) +  (g(u))_{x}  + \sigma(u)\dot W_t,\quad t>0\,; \quad u(0,\cdot) = u_0(\cdot),
\end{equation}
where $L$ is a second order operator, $f(u)$ and $g(u)$ are nonlinear functions, and $\sigma(u)$ is a bounded function. Many studies have been conducted for equation \eqref{burgers_type_equation}; see \cite{bertini1994stochastic,da1994stochastic,da1995stochastic,gyongy1998existence,gyongy1999stochastic,gyongy2000lp,leon2000stochastic,englezos2013stochastic,catuogno2014strong,lewis2018stochastic,rockner2006kolmogorov,crighton1992modern,ladyzhenskaia1968linear,dix1996nonuniqueness,bekiranov1996initial,tersenov2010generalized}. For example, if $f(u) = 0$ and $g(u) = u^2$, the existence and properties of a solution to \eqref{burgers_type_equation} with additive noise are introduced in \cite{bertini1994stochastic,da1994stochastic}. Also, in \cite{da1995stochastic}, similar results are obtained for equation \eqref{burgers_type_equation} with  multiplicative noise. In \cite{leon2000stochastic,lewis2018stochastic}, the regularity and moment estimate of solutions  are achieved. Besides, the unique solvability of stochastic Burger's equation with random coefficients  is contained in \cite{englezos2013stochastic}. Furthermore, the unique solvability of \eqref{burgers_type_equation} with more general $f$ and $g$ is considered in \cite{gyongy1998existence,gyongy1999stochastic,gyongy2000lp}. Especially in \cite{gyongy1998existence}, the unique solvability of \eqref{burgers_type_equation} is provided under the assumption that $f$ is linear and $g$ has quadratic growth. 
For more information, see \cite{rockner2006kolmogorov} and references therein.

One can notice that most of the above results are obtained under the condition that $\sigma(u)$ is bounded. Thus, as another example of a semilinear equation, we suggest an equation having an unbounded diffusion coefficient. Consider a stochastic partial differential equation with nonlinear diffusion coefficient $|u|^\gamma$;
\begin{equation}
\label{second_order_SPDE_example_1}
u_t =  Lu + \xi|u|^{\gamma} \dot W,\quad t>0 \,;\quad u(0,\cdot) = u_0,
\end{equation}
where $L$ is a second order differential operator, $\gamma>0$, and $\dot W$ is a space-time white noise; see \cite{Kijung,kry1999analytic,krylov1997result,krylov1997spde,mueller1991long,mueller2014nonuniqueness,mytnik1998weak,mytnik2011pathwise,walsh1986introduction,Xiong,choi2021regularity,han2019boundary}. Particularly, \cite{kry1999analytic,krylov1997result,krylov1997spde,mueller1991long,choi2021regularity,han2019boundary} describe the solvability of \eqref{second_order_SPDE_example_1} with $\gamma > 1$. In \cite{mueller1991long}, the long-time existence of a mild solution to equation \eqref{second_order_SPDE_example_1} is proved if $L = \Delta$, $\xi = 1$, $\cI = (0,1)$, and $u(0) = u_0$ is deterministic, continuous, and nonnegative. On the other hand, in \cite[Section 8.4]{kry1999analytic}, the existence of a solution in $L_p(\bR)$ spaces is obtained with a second-order differential operator $L = aD^2 + bD + c$, where the bounded coefficients $a,b,c,\xi$ depending on $(\omega,t,x)$. Also, in \cite{han2019boundary}, on a unit interval, interior regularity and boundary behavior of a solution are obtained with $L = aD^2+bD+c$, random and space and time-dependent coefficients $a,b,c,\xi$, and random nonnegative initial data $u(0) = u_0$. In the case of $\gamma > 3/2$, \cite{mueller2000critical} shows that if $L = \Delta$, $\xi = 1$, and nonnegative initial data $u_0$ is nontrivial and vanishing at endpoints, then there is a positive probability that a solution blows up in finite time.

This paper aims to obtain the uniqueness, existence, $L_p$ regularity, and maximal H\"older regularity of the solution to a semilinear stochastic partial differential equation driven by a multiplicative space-time white noise
\begin{equation}
\label{main_equation}
u_t = a u_{xx} + bu_x + cu + |u|^\lambda u_x  + \sigma(u)\dot W_t,\quad t>0\,; \quad u(0,\cdot) = u_0(\cdot),
\end{equation}
where $\lambda > 0$, and $\dot W$ is a space-time white noise. The function $u_0$ is a nonnegative random initial data. The coefficient $a(\omega,t,x),b(\omega,t,x),c(\omega,t,x)$ are $\cP\times\cB(\bR)$-measurable, and $\bar b(\omega,t)$ is $\cP$-measurable. The coefficients $a,b,c,$ and $\bar b$ are uniformly bounded, and the leading coefficient $a$ satisfies the ellipticity condition. The diffusion coefficient $\sigma(\omega,t,x,u)$ is $\cP\times\cB(\bR^d)\times\cB(\bR)$-measurable. According to the conditions on $\lambda$ and $\sigma(u)$, we separate two cases;
\begin{enumerate}[(i)]
\item
$\lambda\in(0,1]$ and $\sigma(u)$ has \textit{bounded} Lipschitz continuity and linear growth in $u$,

\item
$\lambda\in(0,1)$ and  $\sigma(u) = \mu|u|^{1+\lambda_0}$ with $\lambda_0\in[0,1/2)$ (i.e., $\sigma(u)$ is super-linear).
\end{enumerate}
It should be noted that different condition on $\lambda$ is assumed depending on the types of diffusion coefficient $\sigma(u)$. In other words, if $\sigma(u)$ is \textit{bounded} in $u$, $\lambda$ is assumed to be less than or \textit{equal} to $1$; $\lambda\in(0,1]$. On the other hand, if $\sigma(u) = |u|^{1+\lambda_0}$ with $[0,1/2)$, $\lambda$ is considered to be less than $1$; $\lambda\in(0,1)$.

Our result has several advantages. The first novelty of our work is that all the results are obtained with random initial data $u_0(\omega, x)$ and random coefficients $a(\omega,t,x)$, $b(\omega,t,x)$, $c(\omega,t, x)$, $\bar{b}(\omega,t)$, and $\sigma(\omega,t,x,u)$. This accomplishment is one of the benefit of the Krylov's $L_p$-theory, compared with the researches that the coefficients $a, b, c, \bar b$ are assumed to be constants.

Secondly, the regularity of the solution is presented, and the dependence of the regularity of the solution is proposed. For example, if $\lambda\in(0,1]$, $\sigma(u)$ has Lipschitz continuity, linear growth, and boundedness in $u$, and nonnegative initial data $u_0\in U_p^{1/2}\cap L_1(\Omega;L_1(\bR))$ for any $p>2$, then equation \eqref{main_equation} has a solution $u\in \cH_{p,loc}^{1/2}$ (see Definition \ref{definition_of_sol_space}) such that for any $T<\infty$ and small $\ep>0$, almost surely
\begin{equation*}
u \in C^{1/2 - \ep,1/4 - \ep}_{t,x}([0,T]\times\bR).
\end{equation*}
On the other hand, if $\lambda \in (0,1)$, $\sigma(u) = |u|^{1+\lambda_0}$ with $\lambda_0 \in [0,1/2)$, and nonnegative initial data $u_0\in U_p^{1/2-\left(\lambda - 1/2\right)\vee\lambda_0}\cap L_1(\Omega;L_1(\bR))$ for any $p>2$, then equation \eqref{main_equation} has a solution $u\in\cH_{p,loc}^{1/2-(\lambda -1/2)\vee\lambda_0}$ such that for any $T<\infty$ and small $\ep>0$, almost surely 
\begin{equation*}
u \in C^{\frac{1/2-(\lambda -1/2)\vee\lambda_0}{2} - \ep,1/2-\left(\lambda - 1/2\right)\vee\lambda_0 - \ep}_{t,x}([0,T]\times\bR).
\end{equation*}
It should be noted that if the diffusion coefficient $\sigma(u)$ is \textit{bounded}, the regularities of the solution is independent of $\lambda$. However, if the diffusion coefficient $\sigma(u)=|u|^{1+\lambda_0}$, $\lambda$ and $\lambda_0$ affect the solution regularities.

Lastly, we suggest a sufficient condition for the unique solvability in $L_p$ spaces. For example, the coefficient $\bar b$ is considered to be independent of $x$ to handle the nonlinear term $\bar{b} |u|^\lambda u_x$. Also, we assume algebraic conditions on $\lambda$ (and $\lambda_0$ if $\sigma(u) = |u|^{1+\lambda_0}$). Specifically, if $\sigma(u)$ has Lipschitz continuity, linear growth, and \textit{boundedness} in $u$, $\lambda$ is taken to be less than or \textit{equal} to $1$; $\lambda\in(0,1]$. On the other hand, if $\sigma(u) = |u|^{1+\lambda_0}$, $\lambda$ and $\lambda_0$ are expected to satisfy  $\lambda\in(0,1)$ and $\lambda_0\in[0,1/2)$. Indeed, if $\sigma(u)$ has \textit{boundedness} in $u$, we separate the solution $u$ into the noise-related part and the nonlinear-related part. Then, to control the nonlinear-related part, the chain rule and fundamental theorem of calculus are employed with the assumption that $\lambda\in(0,1]$; see Lemma \ref{Lq_bound_2}. However, if $\sigma(u) = |u|^{1+\lambda_0}$, the nonlinear terms $|u|^\lambda u_x$ and $|u|^{1+\lambda_0}$ should be controlled simultaneously. Since the uniform $L_1$ bound of the solution is obtained and the nonlinear terms are interpreted as
$$ u^\lambda u_{x} = \frac{1}{1+\lambda} \left( |u|^{1+\lambda} \right)_{x} = \frac{1}{1+\lambda} \left( |u|^{\lambda}\cdot |u| \right)_{x}\quad\text{and}\quad |u|^{1+\lambda_0} = |u|^{\lambda_0}\cdot |u|,
$$
the surplus parts $|u|^\lambda$ and $|u|^{\lambda_0}$ should be dominated by the uniform $L_1$ bound of the solution in the estimate. Therefore, $|u|^\lambda$ and $|u|^{\lambda_0}$ need to be summable to power $s = 1/\lambda$ and $2s_0 = 1/\lambda_0$ with $s,s_0>1$; see Section \ref{Proof of the second case}. Thus, we assume $\lambda\in(0,1)$ and $\lambda_0\in[0,1/2)$.

This paper is organized as follows. In Section \ref{sec:preliminaries}, preliminary definitions and properties are introduced. Section \ref{sec:main_results} provides the existence, uniqueness, $L_p$-regularity, and maximal H\"older regularity of a solution to equation \eqref{main_equation}. Sections \ref{Proof of the first case} and \ref{Proof of the second case} contain proof of the main results.

We finish introduction with the notations. Let $\bR$ and $\bN$ denote the set of real numbers and natural numbers, respectively. We use $:=$ to denote definition.  For a real-valued function $f$, we define
\begin{equation*}
f^+:= \frac{f+|f|}{2},\quad f^- = -\frac{f-|f|}{2}
\end{equation*} For a normed space $F$, a measure space $(X,\mathcal{M},\mu)$, and $p\in [1,\infty)$, a space $L_{p}(X,\cM,\mu;F)$ is a set of $F$-valued $\mathcal{M}^{\mu}$-measurable function satisfying
\begin{equation}
\label{def_norm}
\| u \|_{L_{p}(X,\cM,\mu;F)} := \left( \int_{X} \| u(x) \|_{F}^{p}\mu(dx)\right)^{1/p}<\infty.
\end{equation}
A set $\mathcal{M}^{\mu}$ is the completion of $\cM$ with respect to the measure $\mu$.
For $\alpha\in(0,1]$ and $T>0$, a set $C^{\alpha}([0,T];F)$ is the set of $F$-valued continuous functions $u$ such that
$$ |u|_{C^{\alpha}([0,T];F)}:=\sup_{t\in[0,T]}|u(t)|_{F}+\sup_{\substack{s,t\in[0,T], \\ s\neq t}}\frac{|u(t)-u(s)|_F}{|t-s|^{\alpha}}<\infty.$$
For $a,b\in \bR$, set
$a \wedge b := \min\{a,b\}$, $a \vee b := \max\{a,b\}$.
Let $\cS = \cS(\bR^d)$ denote the set of Schwartz functions on $\bR^d$. The Einstein's summation convention with respect to $i,j$, and $k$ is assumed throughout the article. A generic constant is denoted as $N$, which varies from line to line, and $N=N(a,b,\ldots)$ denotes that the constant $N$ depends only on $a,b,\ldots$. For functions depending on $\omega$, $t$,  and $x$, the argument $\omega \in \Omega$ is usually omitted.

\section{Preliminaries} 
\label{sec:preliminaries}

This section is devoted to reviewing the definitions and properties of the Brownian sheet $W(A)$ and stochastic Banach spaces $\cH_p^{\gamma+2}(\tau)$. The Brownian sheet $W(A)$ induces space-time white noise $\dot W$, and it is used to interpret equation \eqref{main_equation}. The stochastic Banach space $\cH_p^{\gamma+2}(\tau)$ is employed as a solution space. For more detail, see \cite{kry1999analytic,grafakos2009modern,krylov2008lectures}.

Let $(\Omega, \cF, P)$ be a complete probability space equipped with a filtration $\{\cF_t\}_{t\geq0}$. The filtration $\{\cF_t\}_{t\geq0}$ is assumed to satisfy the usual conditions. A set $\cP$ denotes the predictable $\sigma$-field related to $\{\cF_t\}_{t\geq0}$. 

\begin{definition}[The Brownian sheet]
The Brownian sheet $\{W(A): A\in\cB((0,\infty)\times\bR), \Lambda(A) < \infty \}$ is a centered Gaussian random field and its covariance is given by
\begin{equation*} 
\bE[W(A)W(B)]=\Lambda(A \cap B),
\end{equation*}
where $A,B$ are bounded Borel subsets in $(0,\infty)\times\bR$, and $\Lambda(\cdot)$ is the Lebesgue measure on $(0,\infty)\times\bR$.
\end{definition}

\begin{remark} \label{remark:representation} 

To interpret the noise part of equation \eqref{main_equation}, we employ Walsh's stochastic integral with respect to Gaussian random measure $W(ds,dx)$; see \cite{walsh1986introduction}. It should be mentioned that the stochastic integral with respect to the martingale measure can be written as a series of It\^o's stochastic integral. For example, if $A$ is a bounded Borel measurable subset in $(0,\infty)\times\bR$, we have
\begin{equation} 
\label{equivalent_def}
W(A) = \int_{A} W(ds,dx) = \sum_{k=1}^{\infty}\int_0^\infty \int_{\bR} 1_{A}(s,x) \eta_k(x)dxdw^k_s,
\end{equation}
where $\{w_t^k:k\in\bN\}$ is a set of one-dimensional independent Wiener processes and $\{\eta_k:k\in\bN\}$ is a set of orthonormal basis in $L_2(\bR)$. Without loss of generality, we may assume that $\eta_k$ is bounded for each $k$; e.g. see \cite[Section 8.3]{kry1999analytic}.
\end{remark}

\begin{definition}[Bessel potential space]
Let $p>1$ and $\gamma \in \mathbb{R}$. The space $H_p^\gamma=H_p^\gamma(\bR)$ is the set of all tempered distributions $u$ on $\bR$ such that
$$ \| u \|_{H_p^\gamma} := \left\| (1-\Delta)^{\gamma/2} u\right\|_{L_p} = \left\| \cF^{-1}\left[ \left(1+|\xi|^2\right)^{\gamma/2}\cF(u)(\xi)\right]\right\|_{L_p}<\infty.
$$
Similarly, $H_p^\gamma(\ell_2) = H_p^\gamma(\bR;\ell_2)$ is a space of $\ell_2$-valued functions $g=(g^1,g^2,\cdots)$ such that 
$$ \|g\|_{H_{p}^\gamma(\ell_2)}:= \left\| \left| \left(1-\Delta\right)^{\gamma/2} g\right|_{\ell_2}\right\|_{L_p} = \left\| \left|\cF^{-1}\left[ \left(1+|\xi|^2\right)^{\gamma/2}\cF(g)(\xi)\right]\right|_{\ell_2} \right\|_{L_p}
< \infty. 
$$
When $\gamma = 0$, we write $L_p := H_p^0$ and $L_p(\ell_2) := H_p^0(\ell_2)$.
\end{definition}

\begin{remark} \label{Kernel}
Note that for $\gamma\in (0,\infty)$ and $u\in \cS$, the operator $(1-\Delta)^{-\gamma/2}$ has a representation.
\begin{equation*}
(1-\Delta)^{-\gamma/2}u(x)=\int_{\bR}R_{\gamma}(x-y)u(y)dy,
\end{equation*}
where 
\begin{equation*} 
|R_\gamma(x)| \leq N(\gamma)\left(e^{-\frac{|x|}{2}}1_{|x|\geq2} + A_\gamma(x)1_{|x|<2}\right),
\end{equation*}
and
\begin{equation*}
\begin{aligned}
A_{\gamma}(x):=
\begin{cases}
|x|^{\gamma-1} + 1 + O(|x|^{\gamma+1}) \quad &\mbox{if} \quad 0<\gamma<1,\\ 
\log(2/|x|) + 1 + O(|x|^{2}) \quad &\mbox{if} \quad \gamma=1,\\ 
1 + O(|x|^{\gamma-1}) \quad &\mbox{if} \quad \gamma>1.
\end{cases}
\end{aligned}
\end{equation*}
For more detail, see \cite[Proposition 1.2.5.]{grafakos2009modern}.

\end{remark}

We introduce the space of point-wise multipliers in $H_p^\gamma$.

\begin{definition}
\label{def_pointwise_multiplier}
Fix $\gamma\in\bR$ and $\alpha\in[0,1)$ such that $\alpha = 0$ if $\gamma\in\bZ$ and $\alpha>0$ if $|\gamma|+\alpha$ is not an integer. Define
\begin{equation*}
\begin{aligned}
B^{|\gamma|+\alpha} = 
\begin{cases}
B(\bR) &\quad\text{if } \gamma = 0, \\
C^{|\gamma|-1,1}(\bR) &\quad\text{if $\gamma$ is a nonzero integer}, \\
C^{|\gamma|+\alpha}(\bR) &\quad\text{otherwise};
\end{cases}
\end{aligned}
\end{equation*}
\begin{equation*}
\begin{aligned}
B^{|\gamma|+\alpha}(\ell_2) = 
\begin{cases}
B(\bR,\ell_2) &\quad\text{if } \gamma = 0, \\
C^{|\gamma|-1,1}(\bR,\ell_2) &\quad\text{if $\gamma$ is a nonzero integer}, \\
C^{|\gamma|+\alpha}(\bR,\ell_2) &\quad\text{otherwise},
\end{cases}
\end{aligned}
\end{equation*}
where $B(\bR)$ is the space of bounded Borel functions on $\bR$, $C^{|\gamma|-1,1}(\bR)$ is the space of $|\gamma|-1$ times continuous differentiable functions whose derivatives of $(|\gamma|-1)$-th order derivative are Lipschitz continuous, and $C^{|\gamma|+\alpha}$ is the real-valued H\"older spaces. The space $B(\ell_2)$ denotes a function space with $\ell_2$-valued functions, instead of real-valued function spaces.

\end{definition}

Below we gather the properties of $H_p^\gamma$.

\begin{lemma}
\label{prop_of_bessel_space} 
Let $p>1$ and $\gamma \in \bR$. 
\begin{enumerate}[(i)]
\item 
\label{dense_subset_bessel_potential}
The space  $C_c^\infty(\bR)$ is dense in $H_{p}^{\gamma}$. 

\item
\label{sobolev-embedding} 
Let $\gamma - 1/p = n+\nu$ for some $n=0,1,\cdots$ and $\nu\in(0,1]$. Then, for any  $i\in\{ 0,1,\cdots,n \}$, we have
\begin{equation} 
\label{holder embedding}
\left| D^i u \right|_{C(\bR)} + \left[D^n u\right]_{\cC^\nu(\bR)} \leq N \| u \|_{H_{p}^\gamma},
\end{equation}
where $N = N(p,\gamma)$ and $\cC^\nu$ is a Zygmund space.

\item
\label{bounded_operator}
The operator $D_i:H_p^{\gamma}\to H_p^{\gamma+1}$ is bounded. Moreover, for any $u\in H_p^{\gamma+1}$,
$$ \left\| D^i u \right\|_{H_p^\gamma} \leq N\| u \|_{H_p^{\gamma+1}},
$$
where $N = N(p,\gamma)$.

\item
\label{norm_bounded}
Let $\mu\leq\gamma$ and $u\in H_p^\gamma$. Then $u\in H_p^\mu$ and 
$$ \| u \|_{H_p^\mu} \leq \| u \|_{H_p^\gamma}. 
$$

\item 
\label{iso} (isometry). For any $\mu,\gamma\in\bR$, the operator $(1-\Delta)^{\mu/2}:H_p^\gamma\to H_p^{\gamma-\mu}$ is an isometry.

\item
\label{multi_ineq} (multiplicative inequality). Let 
\begin{equation*} \label{condition_of_constants_interpolation}
\begin{gathered}
\ep\in[0,1],\quad p_i\in(1,\infty),\quad\gamma_i\in \bR,\quad i=0,1,\\
\gamma=\ep\gamma_0+(1-\ep)\gamma_1,\quad1/p=\ep/p_0+(1-\ep)/p_1.
\end{gathered}
\end{equation*}
Then, we have
\begin{equation*}
\|u\|_{H^\gamma_{p}} \leq \|u\|^{\ep}_{H^{\gamma_0}_{p_0}}\|u\|^{1-\ep}_{H^{\gamma_1}_{p_1}}.
\end{equation*}

\item \label{pointwise_multiplier}
Let $u\in H_p^\gamma$. Then, we have
\begin{equation*}
\| au \|_{H_p^\gamma} \leq N\| a \|_{B^{|\gamma|+\alpha}}\| u \|_{H_p^\gamma}\quad\text{and}\quad\| bu \|_{H_p^\gamma(\ell_2)} \leq N\| b \|_{B^{|\gamma|+\alpha}(\ell_2)}\| u \|_{H_p^\gamma},
\end{equation*}
where $N = N(\gamma,p)$ and $B^{|\gamma|+\alpha},B^{|\gamma|+\alpha}(\ell_2)$ are introduced in Definition \ref{def_pointwise_multiplier}.
\end{enumerate}
\end{lemma}
\begin{proof}
The above results are well-known; for instance, for \eqref{dense_subset_bessel_potential}-\eqref{multi_ineq}, see Theorem 13.3.7 (i), Theorem 13.8.1, Theorem 13.3.10, Corollary 13.3.9, Theorem 13.3.7 (ii), Exercise 13.3.20 of \cite{krylov2008lectures}, respectively. 
For \eqref{pointwise_multiplier}, see  \cite[Lemma 5.2]{kry1999analytic}.
\qed
\end{proof}

Now stochastic Banach spaces and solution spaces are provided. For more detail, see Section 3 of \cite{kry1999analytic}.

\begin{definition}[Stochastic Banach spaces]
Let $\tau\leq T$ be a bounded stopping time, $p>1$ and $\gamma\in\bR$. Define
\begin{gather*}
\opar0,\tau\cbrk:=\{ (\omega,t):0<t\leq \tau(\omega) \},\\
\mathbb{H}_{p}^{\gamma}(\tau) := L_p\left(\opar0,\tau\cbrk, \mathcal{P}, dP \times dt ; H_{p}^\gamma\right),\\
\mathbb{H}_{p}^{\gamma}(\tau,\ell_2) := L_p\left(\opar0,\tau\cbrk,\mathcal{P}, dP \times dt;H_{p}^\gamma(\ell_2)\right),\\
U_{p}^{\gamma} :=  L_p\left(\Omega,\cF_0, dP ; H_{p}^{\gamma-2/p}\right).
\end{gather*}
If $\gamma = 0$, we use notation $\bL$ instead of $\bH$. For example, $\mathbb{L}_{p}(\tau) = \mathbb{H}_{p}^{0}(\tau)$. Also, we define the norm of above space as in \eqref{def_norm}. 
\end{definition}

\begin{definition}[Solution spaces] 
\label{definition_of_sol_space}
Let $\tau\leq T$ be a bounded stopping time and $p\geq2$. 

\begin{enumerate}[(i)]
\item 
For $u\in \bH_p^{\gamma+2}(\tau)$, we write $u\in\cH^{\gamma+2}_p(\tau)$ if there exists $u_0\in U_{p}^{\gamma+2}$ and  $(f,g)\in
\bH_{p}^{\gamma}(\tau)\times\bH_{p}^{\gamma+1}(\tau,\ell_2)$ such that
\begin{equation*}
du = fdt+\sum_{k=1}^{\infty} g^k dw_t^k,\quad   t\in (0, \tau]\,; \quad u(0,\cdot) = u_0
\end{equation*}
in the sense of distributions. In other words, for any $\phi\in \cS$, the equality
\begin{equation} \label{def_of_sol}
(u(t,\cdot),\phi) = (u_0,\phi) + \int_0^t(f(s,\cdot),\phi)ds + \sum_{k=1}^{\infty} \int_0^t(g^k(s,\cdot),\phi)dw_s^k
\end{equation}
holds for all $t\in [0,\tau]$ almost surely. 

\item

The norm of the function space $\cH_{p}^{\gamma+2}(\tau)$ is defined as
\begin{equation}
\label{def_of_sol_norm}
\| u \|_{\cH_{p}^{\gamma+2}(\tau)} :=  \| u \|_{\mathbb{H}_{p}^{\gamma+2}(\tau)} + \| f \|_{\mathbb{H}_{p}^{\gamma}(\tau)} + \| g \|_{\mathbb{H}_{p}^{\gamma+1}(\tau,\ell_2)} + \| u_0 \|_{U_{p}^{\gamma+2}}.
\end{equation}

\item
When $\gamma + 2 = 0$, we write $\cL_p(\tau) := \cH_p^0(\tau)$.

\item \label{def_of_local_sol_space}
For a stopping time $\tau \in [0,\infty]$, we say $u \in \cH_{p,loc}^{\gamma+2}(\tau)$ if there exists a sequence of bounded stopping times $\{ \tau_n : n\in\bN \}$ such that $\tau_n\uparrow \tau$ (a.s.) as $n\to\infty$ and $u\in \cH_{p}^{\gamma+2}(\tau_n)$ for each $n$. The stopping time $\tau$ is omitted if $\tau = \infty$. We write $u = v$ in $\cH_{p,loc}^{\gamma+2}(\tau)$ if there exists a sequence of bounded stopping times $\{ \tau_n : n\in\bN \}$ such that $\tau_n\uparrow\tau$ (a.s.) as $n\to\infty$ and $u = v$ in $\cH_{p}^{\gamma+2}(\tau_n)$ for each $n$. 
\end{enumerate}

\end{definition}

\begin{remark}
Let $p\geq 2$ and $\gamma\in \bR$. For any $g\in \bH^{\gamma+1}_{p}(\tau,\ell_2)$,  the series of stochastic integral in \eqref{def_of_sol} converges uniformly in $t$ in probability on $[0,\tau \wedge T]$ for any $T$. Therefore, $(u(t,\cdot),\phi)$ is continuous in $t$ (See, e.g., \cite[Remark 3.2]{kry1999analytic}).
\end{remark}

\begin{remark}
\label{equation_sense_remark}
A stochastic partial differential equation driven by space-time white noise is understood in the sense of distribution. For example, for a bounded stopping time $\tau\leq T$, consider 
\begin{equation} \label{example_eq_STWN}
u_t =  u_{xx} + u \dot{W}_t, \quad (t,x)\in(0,\tau)\times \bR \,; \quad u(0,\cdot) = u_0.
\end{equation}
We interpret equation \eqref{example_eq_STWN} as follows: for any $\phi\in\cS$, equality
\begin{equation*}
\begin{aligned}
(u(t,\cdot),\phi) &= (u_0,\phi) + \int_0^t\int_{\bR} u(s,x)\phi_{xx}(x) dxds 
 + \int_0^t\int_{\bR} u(s,x)\phi(x)W(ds,dx)
\end{aligned}
\end{equation*}
holds for all $t\leq \tau$ almost surely. By the way, due to Remark \ref{remark:representation}, the Walsh's integral with respect to $W(ds,dx)$ can be written as the series of It\^o stochastic integral: 
$$ \int_{0}^t\int_{\bR}  u(s,x) \phi(x)\,W(ds,dx) = \sum_{k=1}^{\infty}\int_0^t\int_{\bR}  u(s,x) \phi(x) \eta_k(x) dxdw^k_s 
$$
holds for all $t\leq\tau$ almost surely. Thus, equation \eqref{example_eq_STWN} means for any $\phi\in \cS$, the equality
\begin{equation*} 
\begin{aligned}
(u(t,\cdot),\phi) &= (u_0,\phi) + \int_0^t\int_{\bR} u(s,x)\phi_{xx}(x) dxds  +  \int_0^t\int_{\bR} u(s,x)\phi(x)\eta_k(x) dxdw_s^k
\end{aligned}
\end{equation*}
holds for all $t\in [0,\tau]$ almost surely. Therefore, we consider 
\begin{equation*}
du = u_{xx} dt +  u \eta_k dw^k_t,
\quad (t,x)\in(0,\tau)\times \bR \,; \quad u(0,\cdot) = u_0.
\end{equation*}
\end{remark}

\vspace{2mm}

Next the properties of the solution space $\cH_p^{\gamma+2}(\tau)$ are introduced.

\begin{theorem} 
\label{embedding}
Let $\tau\leq T$ be a bounded stopping time.
\begin{enumerate}[(i)]

\item \label{completeness}
For any $p\geq2$, $\gamma\in\bR$, $\cH_p^{\gamma+2}(\tau)$ is a Banach space with the norm $\| \cdot \|_{\cH_p^{\gamma+2}(\tau)}$.

\item \label{large-p-embedding} 
If $p>2$, $\gamma\in\bR$, and $1/p < \alpha < \beta < 1/2$, then for any  $u\in\cH_{p}^{\gamma+2}(\tau)$, we have $u\in C^{\alpha-1/p}\left([0,\tau];H_{p}^{\gamma+2-2\beta}\right)$ (a.s.) and
\begin{equation} 
\label{solution_embedding}
\mathbb{E}| u |^p_{C^{\alpha-1/p}([0,\tau];H_{p}^{\gamma + 2 - 2\beta} )} \leq N(d,p,\gamma,\alpha,\beta,T)\| u \|^p_{\cH_{p}^{\gamma+2}(\tau)}.
\end{equation}

\item \label{gronwall_type_ineq} 
Let $p > 2$, $\gamma\in\bR$, and $u\in\cH_p^{\gamma+2}(\tau)$. If there exists $\gamma_0 < \gamma$ such that
\begin{equation}
\label{gronwall_type_ineq_condi}
\| u \|_{\cH_p^{\gamma+2}(\tau\wedge t)}^p \leq N_0 + N_1 \| u \|_{\bH_p^{\gamma_0+2}(\tau\wedge t)}^p
\end{equation}
for all $t\in(0,T)$, then we have
\begin{equation} \label{modified_Gronwall}
\| u \|_{\cH_p^{\gamma+2}(\tau\wedge T)}^p \leq N_0N.
\end{equation}
where $N =  N(N_1,p,\gamma,T)$.

\end{enumerate}
\end{theorem}
\begin{proof}

For \eqref{completeness} and \eqref{large-p-embedding}, we refer the reader \cite[Theorem 3.7]{kry1999analytic} and \cite[Theorem 7.2]{kry1999analytic}. To show \eqref{gronwall_type_ineq}, apply Lemma \ref{prop_of_bessel_space} \eqref{multi_ineq} to \eqref{gronwall_type_ineq_condi}. Then, we have
\begin{equation*}
\begin{aligned}
\|u\|_{\cH_p^{\gamma+2}(\tau\wedge t)}^p &\leq N_0 +  N_1 \| u  \|_{\bH^{\gamma_0+2}_p(\tau\wedge t)}^p \\
&\leq N_0 +  \frac{1}{2} \bE\int_0^{\tau\wedge t} \| u(s,\cdot) \|_{H^{\gamma+2}_p}^p ds + N \bE\int_0^{\tau\wedge t} \| u(s,\cdot) \|_{H^{\gamma}_p}^p ds\\
&\leq N_0 +  \frac{1}{2}  \| u \|_{\cH^{\gamma+2}_p(\tau\wedge t)}^p + N \int_0^{t} \bE 1_{s\leq \tau}\| u(s,\cdot) \|_{H^{\gamma}_p}^p ds\\
&\leq N_0 + \frac{1}{2}  \| u \|_{\cH^{\gamma+2}_p(\tau\wedge t)}^p + N \int_0^{t} \bE \sup_{r\leq \tau\wedge s}\| u(r,\cdot) \|_{H^{\gamma}_p}^p ds, \\
\end{aligned}
\end{equation*}
where $N = N(N_1,p,\gamma)$.
By removing $\frac{1}{2}  \| u \|_{\cH^{\gamma+2}_p(\tau\wedge t)}^p$ both sides and applying \eqref{solution_embedding}, we have
\begin{equation*}
\begin{aligned}
\|u\|_{\cH_p^{\gamma+2}(\tau\wedge t)}^p
&\leq 2N_0 + 2N \int_0^{t} \| u \|_{\cH^{\gamma+2}_p(\tau\wedge  s)}^p ds,
\end{aligned}
\end{equation*}
where $N = N(N_1,p,\gamma,T)$. By the Gr\"onwall's inequality, we have \eqref{modified_Gronwall}. The theorem is proved.
\qed
\end{proof}

\begin{corollary} 
\label{embedding_corollary}
Let $\tau\leq T$ be a bounded stopping time and $\kappa \in (0,1/2)$. Suppose $p\in(2,\infty)$, $\alpha,\beta\in(0,\infty)$ satisfy
\begin{equation} \label{condition_for_alpha_beta}
\frac{1}{p} < \alpha < \beta <  \frac{1}{2}\left(\frac{1}{2}-\kappa-\frac{1}{p}\right).
\end{equation}
Then, for any $\delta \in \left[0, \frac{1}{2}-\kappa-2\beta-\frac{1}{p}\right)$,
we have $u\in C^{\alpha-1/p}([0,\tau];C^{1/2-\kappa-2\beta-1/p-\delta} )$ (a.s.) and
\begin{equation} \label{sol_embedding}
\mathbb{E} |u|^p_{C^{\alpha-1/p}([0,\tau];C^{1/2-\kappa-2\beta-1/p-\delta} )}\leq N\|u\|_{\cH^{1/2-\kappa}_{p} (\tau)}^p,
\end{equation}
where $N = N(p,\alpha,\beta,\kappa,T)$. 
\end{corollary}
\begin{proof}
Set $\gamma = 1/2-\kappa-2\beta-\delta$. Then, by Lemma \ref{prop_of_bessel_space} \eqref{sobolev-embedding},  we have
\begin{equation} 
\label{holder_lp_embedding}
\begin{aligned}
|u(t,\cdot)|_{C^{1/2-\kappa-2\beta-1/p-\delta}}&\leq N\| u(t,\cdot) \|_{H_{p}^{1/2-\kappa-2\beta-\delta}} \leq N\| u(t,\cdot) \|_{H_{p}^{1/2-\kappa-2\beta}}
\end{aligned}
\end{equation}
for all $t \in [0,\tau]$ almost surely.
By \eqref{holder_lp_embedding} and \eqref{solution_embedding}, we have \eqref{sol_embedding}. The corollary is proved.
\qed
\end{proof}


\section{Main results} 
\label{sec:main_results}
This section provides the uniqueness, existence, $L_p$ regularity, and maximal H\"older regularity of the solution to semilinear equation
\begin{equation}
\label{burger's_eq_space_time_white_noise}
du = \left(au_{xx} + b u_{x} + cu + \bar{b} |u|^{\lambda } u_{x}\right) dt +  \sigma(u)\eta_k dw_t^k\,; \quad u(0,\cdot) = u_0(\cdot)
\end{equation}
on $(t,x)\in(0,\infty)\times\bR$ with $\lambda>0$. The coefficients $a,b,c$ are $\cP\times \cB(\bR)$-measurable, $\bar b$ is $\cP$-measurable, and $a,b,c,\bar b$ are uniformly bounded; see Assumption \ref{deterministic_part_assumption_on_coeffi_space-time_white_noise}. Notice that $\bar b$ is assumed to be independent of $x$ to control the nonlinear term $\bar b |u|^\lambda u_x$ in the estimate; see Remarks \ref{bar b is independent of x} and \ref{bar b is independent of x 2}. The set of bounded functions $\{ \eta_k : k\in\bN \}$ is an orthonormal $L_2(\bR)$ basis introduced in Remark \ref{remark:representation}. 

Based on the conditions of nonlinear terms $\bar b |u|^\lambda u_x$ and $\sigma(u)$, we separate two cases. Each case is discussed in Sections \ref{the first case} and \ref{the second case}; 
\begin{enumerate}[\text{Case} 1)]
\item $\lambda\in(0,1]$ and $\sigma(u)$ has Lipschitz continuity, linear growth, and \textit{boundedness} in $u$,

\item $\lambda\in(0,1)$ and $\sigma(u) = |u|^{1+\lambda_0}$ with $\lambda_0\in[0,1/2)$.
\end{enumerate}
Notice that the condition on $\lambda$ is changed according to the types of diffusion coefficient $\sigma(u)$. In other words, if $\sigma(u)$ is \textit{bounded} (Case 1), we consider $\lambda$ is less than or \textit{equal} to $1$; $\lambda\in(0,1]$. On the other hand, if $\sigma(u)$ is unbounded (Case 2), $\lambda$ is assumed to less than $1$; $\lambda \in (0,1)$. 

Besides, the regularity of the solution varies in each case. For example, if $\sigma(u)$ is bounded (Case 1) and $u_0\in \cap_{p>2}U_p^{1/2}\cap L_1(\Omega;L_1)$, then solution $u$ introduced in Corollary \ref{maximal holder regularity1} satisfies for $T<\infty$ and small $\ep>0$,
$$ \sup_{t\leq T}|u(t,\cdot)|_{C^{1/2-\ep}(\bR)} + \sup_{x\in\bR}|u(\cdot,x)|_{C^{1/4-\ep}([0,T])}  < \infty\quad(a.s.).
$$
On the other hand, if $\sigma(u)$ is unbounded (Case 2) and $ u_0\in \cap_{p>2}U_p^{1/2-(\lambda - 1/2)\vee \lambda_0}\cap L_1(\Omega;L_1)$, solution $u$ introduced in Corollary \ref{maximal holder regularity2} satisfies for $T<\infty$ and small $\ep>0$, 
$$\sup_{t\leq T}|u(t,\cdot)|_{C^{1/2-(\lambda - 1/2)\vee \lambda_0-\ep}(\bR)} + \sup_{x\in\bR}|u(\cdot,x)|_{C^{\frac{1/2-(\lambda - 1/2)\vee \lambda_0}{2}-\ep}([0,T])}  < \infty\quad (a.s.).
$$
Notice that the solution regularity independent of $\lambda$ if $\sigma(u)$ is bounded. On the contrary, the regularity of solution depends on $\lambda$ and $\lambda_0$ if $\sigma(u)$ is super-linear.

\vspace{1mm}

\subsection{ \it The first case: generalized Burgers' equation with the bounded Lipschitz diffusion coefficient \texorpdfstring{$\sigma(u)$}{Lg}}
\label{the first case}

This section contains results of equation \eqref{burger's_eq_space_time_white_noise} with $\lambda\in(0,1]$. The diffusion coefficient $\sigma(u)$ has Lipschitz continuity, linear growth, and \textit{boundedness} in $u$; see Assumption \ref{stochastic_part_assumption_on_coeffi_space_time_white_noise_large_lambda}. The $L_p$-solvability of equation \eqref{burger's_eq_space_time_white_noise} is provided in Theorem \ref{theorem_burgers_eq_large_lambda_STWN}. Also, we obtain the H\"older regularity of the solution by applying the H\"older embedding theorem for $\cH_p^{\gamma+2}(\tau)$ (Corollary \ref{embedding_corollary}). Furthermore, by employing the uniqueness of the solution in $p$ (Theorem \ref{uniqueness_in_p_1}), we achieve the maximal H\"older regularity of the solution; see Corollary \ref{maximal holder regularity1}.

\vspace{1mm}

The following are assumptions on coefficients to equation \eqref{burger's_eq_space_time_white_noise}. 
Note that $\sigma(u)$ is bounded in Assumption \ref{stochastic_part_assumption_on_coeffi_space_time_white_noise_large_lambda}.

\begin{assumption} 
\label{deterministic_part_assumption_on_coeffi_space-time_white_noise}
\begin{enumerate}[(i)]

\item 
The coefficients $a = a(t,x)$, $b = b(t,x)$, and $c = c(t,x)$ are $\cP\times\cB(\bR)$-measurable.

\item
The coefficient $\bar{b} = \bar{b}(t)$ is predictable.

\item 
There exists $K>0$ such that 
\begin{equation} 
\label{ellipticity_of_leading_coefficients_space-time_white_noise}
K^{-1} \leq a(t,x) \leq  K\quad \text{for all}\quad (\omega,t,x)\in\Omega\times[0,\infty)\times\bR,
\end{equation}
and
\begin{equation} 
\label{boundedness_of_deterministic_coefficients_space-time_white_noise} 
\left| a(t,\cdot) \right|_{C^{2}(\bR)} + \left| b(t,\cdot) \right|_{C^{2}(\bR)} + |c(t,\cdot)|_{C^{2}(\bR)} + \left| \bar{b}(t) \right| \leq K
\end{equation}
for all $(\omega,t)\in\Omega\times[0,\infty)$.
\end{enumerate}

\end{assumption}

\begin{assumption} 
\label{stochastic_part_assumption_on_coeffi_space_time_white_noise_large_lambda}
\begin{enumerate}[(i)]
\item
The function $\sigma(t,x,u)$ is $\cP\times\cB(\bR)\times\cB(\bR)$-measurable.

\item
There exists $K>0$ such that for $\omega\in\Omega$, $t>0$, $x,u,v\in\bR$,
\begin{equation}
\label{boundedness_of_stochastic_coefficients_large_lambda_STWN}
|\sigma(t,x,u) - \sigma(t,x,v)| \leq K|u-v|\quad\text{and}\quad |\sigma(t,x,u)| \leq K(|u|\wedge1).
\end{equation}

\end{enumerate}
\end{assumption}

Now we provide the $L_p$-solvability of equation \eqref{burger's_eq_space_time_white_noise}.

\begin{theorem} 
\label{theorem_burgers_eq_large_lambda_STWN}
Let $\lambda\in(0,1]$. Assume $p>2$ and $\kappa\in(0,1/2)$ satisfy
$$ p-1+\lambda\in2\bN\quad\text{and}\quad p > \frac{6}{1-2\kappa}.
$$ 
Suppose Assumptions \ref{deterministic_part_assumption_on_coeffi_space-time_white_noise} and \ref{stochastic_part_assumption_on_coeffi_space_time_white_noise_large_lambda} hold. For the nonnegative initial data $u_0\in U_p^{1/2-\kappa}\cap L_{1}(\Omega;L_1)$, equation 
\eqref{burger's_eq_space_time_white_noise}
has a unique nonnegative solution $u$ in $\cH_{p,loc}^{1/2-\kappa}$. 
Furthermore, if $\alpha$ and $\beta$ satisfy \eqref{condition_for_alpha_beta}, then for any $T\in (0,\infty)$ and small $\ep>0$, almost surely 
\begin{equation} 
\label{holder_regularity_large_lambda_STWN}
\|u\|^p_{C^{\alpha-1/p}([0,T];C^{1/2-\kappa-2\beta-1/p}(\bR) )} < \infty.
\end{equation}
\end{theorem}
\begin{proof}
See {\bf{Proof of Theorem \ref{theorem_burgers_eq_large_lambda_STWN}}} in Section \ref{Proof of the first case}.
\qed
\end{proof}

\begin{remark}
\label{bar b is independent of x}
The coefficient $\bar b$ is assumed to be independent of $x$, and it is a sufficient condition for the existence of a global solution. In fact, we control the nonlinear term $|u|^\lambda u_x$ with the bound of $\sigma(u)$ and the initial data $u_0$ to show the local existence time blows up. Thus, in the $L_p$ estimate, the chain rule and fundamental theorem of calculus are employed to apply Gr\"onwall's inequality. Here, the coefficient $\bar b$ should be taken out of the integral with respect to $x$ to use the fundamental theorem of calculus. Thus, $\bar b$ is assumed to be independent of $x$; see Lemma \ref{Lq_bound_2}.

\end{remark}

\begin{remark}
\begin{enumerate}[(i)]
\item There exist $\alpha$ and $\beta$ satisfying \eqref{condition_for_alpha_beta} since $p > \frac{6}{1-2\kappa}$.

\item The summability $p$ is assumed to satisfy $p-1+\lambda\in2\bN$ to employ the fundamental theorem of calculus; see \eqref {bound_of_nonlienar_term_STWN_1} of Lemma \ref{Lq_bound_2}.

\item The condition on the initial data $u_0\in L_1(\Omega;L_1)$ is assumed to obtain $L_1$ bound of the solution.
\end{enumerate}
\end{remark}

\begin{remark}

The regularity of the solution is independent of $\lambda$. Indeed, the nonlinear term $|u|^\lambda u_x$ does not affect the regularity of the solution when we extend the local solutions to a global one. More precisely, we prove that the regularity of the local solution is independent of $\lambda$ and the local solution is non-explosive. To prove non-explosive property, it suffices to show that there is a uniform $L_p$ bounded of the local solutions since we employ Krylov's $L_p$ theory; see \cite[Section 8]{kry1999analytic}. During the proof, we separate the local solution into two parts to obtain the uniform $L_p$ bound of the local solution; noise-related part and nonlinear-related part. In Lemma \ref{noise_cancelling_lemma}, we show that the regularity of the noise-related part is mainly affected by space-time white noise and independent of $\lambda$. Also, it turns out that for any $\lambda\in(0,1]$, the $L_p$ bound of nonlinear-related part is dominated by noise-related part; see Lemma \ref{Lq_bound_2}. Thus, we achieve a uniform $L_p$ bound of the local solutions independent of $\lambda$.

\end{remark}

\begin{remark}
The H\"older regularity of the solution introduced in Theorem \ref{theorem_burgers_eq_large_lambda_STWN} depends on $\alpha$ and $\beta$. For example, for small $\ep > 0$, set $\alpha = \frac{1}{p}+\frac{\ep}{4}$ and $\beta = \frac{1}{p}+\frac{\ep}{2}$. Then, we have
\begin{equation*}
\sup_{t\leq T}|u(t,\cdot)|_{C^{\frac{1}{2}-\kappa-\frac{3}{p}-\ep}(\bR)} < \infty\quad(a.s.).
\end{equation*}
Similarly, for small $\ep > 0$, let $\alpha = \frac{1}{2}\left(\frac{1}{2}-\kappa-\frac{1}{p}\right)-\ep$ and $\beta = \frac{1}{2}\left(\frac{1}{2}-\kappa-\frac{1}{p}\right)-\frac{\ep}{2}$. Then, we have
\begin{equation*}
\sup_{x\in\bR}|u(\cdot,x)|_{C^{\frac{1}{2}-\kappa-\frac{3}{2p}-\ep}([0,T])}  < \infty\quad (a.s.).
\end{equation*}

\end{remark}

To obtain the maximal H\"older regularity of the solution, the parameter $p$ should be large.
The following theorem provides the uniqueness of the solution in $p$.

\begin{theorem}
\label{uniqueness_in_p_1}
Assume that all the conditions of Theorem \ref{theorem_burgers_eq_large_lambda_STWN} holds. Let $u\in \cH_{p,loc}^{1/2-\kappa}$ be the solution of equation \eqref{burger's_eq_space_time_white_noise} introduced in Theorem \ref{theorem_burgers_eq_large_lambda_STWN}. If $q>p$ and $u_0\in U_{q}^{1/2-\kappa}\cap L_{1}(\Omega;L_{1})$, then $u\in \cH_{q,loc}^{1/2-\kappa}$.
\end{theorem}
\begin{proof}
See {\bf{Proof of Theorem \ref{uniqueness_in_p_1}}} in Section \ref{Proof of the first case}.
\qed
\end{proof}

By combining Theorems \ref{theorem_burgers_eq_large_lambda_STWN} and \ref{uniqueness_in_p_1}, we obtain the maximal H\"older regularity of the solution.

\begin{corollary}
\label{maximal holder regularity1}
Suppose $u_0\in U_p^{1/2}\cap L_1(\Omega;L_1)$ for all $p>2$. Then, for $T>0$ and small $\ep>0$, we have
\begin{equation*}
\sup_{t\leq T}|u(t,\cdot)|_{C^{1/2-\ep}(\bR)} + \sup_{x\in\bR}|u(\cdot,x)|_{C^{1/4-\ep}([0,T])}  < \infty
\end{equation*}
almost surely.

\end{corollary}
\begin{proof}

Let $T\in(0,\infty)$ and small $\ep>0$. For each $p > \frac{6}{1-\ep/2}$, Theorem \ref{theorem_burgers_eq_large_lambda_STWN} with $\kappa = \ep/4$ yields that there exists a unique solution $u = u_p\in \cH_p^{1/2-\ep/4}(T)$ to equation \eqref{burger's_eq_space_time_white_noise}. Since $u_0 \in U_p^{1/2-\ep/4}$ for all $p > 2$, Theorem \ref{uniqueness_in_p_1} implies that all the solutions $u = u_p$ coincides. Then, by letting $\delta = 0$, $\kappa = \frac{\ep}{4}$, $p = \frac{12}{\ep}$, $\alpha = \frac{1}{p} + \frac{\ep}{8}$, and $\beta = \frac{1}{p} + \frac{\ep}{4}$ in Corollary \ref{embedding_corollary}, almost surely 
\begin{equation*}
\sup_{t\leq T} |u(t,\cdot)|_{C^{\frac{1}{2}-\ep}(\bR)}^p \leq |u|_{C^{\alpha-\frac{1}{p}}\left([0,T];C^{\frac{1}{2}-\kappa-2\beta-\frac{1}{p}}(\bR)\right)}^p<\infty.
\end{equation*}
On the other hand, set $\delta = 0$, $\kappa = \frac{\ep}{4}$, $p = \frac{12}{5\ep}$, $\alpha = \frac{1}{2}\left( \frac{1}{2}-\kappa-\frac{1}{p} \right) - \frac{\ep}{4}$, and $\beta = \frac{1}{2}\left( \frac{1}{2}-\kappa-\frac{1}{p} \right) - \frac{\ep}{2}$ in Corollary \ref{embedding_corollary}. Then, almost surely 
\begin{equation*}
\sup_{x\in \bR} |u(\cdot,x)|_{C^{\frac{1}{4} - \ep}([0,T])}\leq |u|_{C^{\alpha-\frac{1}{p}}\left([0,T];C^{\frac{1}{4}-\frac{\kappa}{2}-2\beta-\frac{1}{p}}(\bR)\right)}^p<\infty.
\end{equation*}
The corollary is proved.

\qed
\end{proof}

\subsection{\it The second case: modified Burgers' equation with the super-linear diffusion coefficient \texorpdfstring{$\sigma(u)$}{Lg}}
\label{the second case}

In this section, we consider
\begin{equation}
\label{burger's_eq_space_time_white_noise_super_linear}
du = \left(au_{xx} + b u_{x} + cu + \bar{b} |u|^{\lambda } u_{x}\right) dt +  \mu|u|^{1+\lambda_0}\eta_k dw_t^k; \quad u(0,\cdot) = u_0(\cdot),
\end{equation}
on $(t,x)\in(0,\infty)\times\bR$ with $\lambda\in(0,1)$ and $\lambda_0\in[0,1/2)$. Theorem \ref{theorem_burgers_eq_STWN_super_linear} contains the uniqueness, existence, and regularity of the solution to equation \eqref{burger's_eq_space_time_white_noise_super_linear}. Besides, the H\"older regularity of the solution follows from the H\"older embedding theorem for $\cH_p^{\gamma+2}(\tau)$ (Corollary \ref{embedding_corollary}). Similarly to Section \ref{the first case}, by considering large $p$, we obtain the maximal H\"older regularity of the solution; see Corollary \ref{maximal holder regularity2}.

\vspace{1mm}

Below is the assumption for the coefficient $\mu$ of \eqref{burger's_eq_space_time_white_noise_super_linear}.

\begin{assumption} 
\label{stochastic_part_assumption_on_coeffi_space_time_white_noise_small_lambda}
\begin{enumerate}[(i)]

\item 
The coefficient $\mu(t,x)$ is $\cP\times\cB(\bR)$-measurable.

\item There exists $K>0$ such that for $\omega\in\Omega, t>0, x\in\bR$,
\begin{equation}
\label{boundedness_of_stochastic_coefficients_small_lambda_STWN}
|\mu(t,\cdot)|_{C(\bR)} \leq K.
\end{equation}
\end{enumerate}
\end{assumption}

We introduce the $L_p$-solvability results on equation \eqref{burger's_eq_space_time_white_noise_super_linear}.

\begin{theorem}
\label{theorem_burgers_eq_STWN_super_linear}

Let $\lambda \in(0,1)$, $\lambda_0\in[0,1/2)$, $1/2 > \kappa > (\lambda -1/2)\vee\lambda_0$, and $p > \frac{6}{1-2\kappa}$. Suppose Assumptions \ref{deterministic_part_assumption_on_coeffi_space-time_white_noise} and \ref{stochastic_part_assumption_on_coeffi_space_time_white_noise_small_lambda} hold. For nonnegative initial data $u_0\in U_p^{1/2-\kappa}\cap L_{1}(\Omega;L_1)$, equation \eqref{burger's_eq_space_time_white_noise_super_linear}
has a unique nonnegative solution $u$ in $\cH_{p,loc}^{1/2-\kappa}$. Furthermore, if $\alpha$ and $\beta$ satisfy \eqref{condition_for_alpha_beta}, then for any $T\in (0,\infty)$ and small $\ep>0$, almost surely 
\begin{equation} 
\label{holder_regularity_large_lambda_STWN_2}
\|u\|^p_{C^{\alpha-1/p}([0,T];C^{1/2-\kappa-2\beta-1/p}(\bR) )} < \infty.
\end{equation}
\end{theorem}
\begin{proof}
See {\bf{Proof of Theorem \ref{theorem_burgers_eq_STWN_super_linear}}} in Section \ref{Proof of the second case}.
\qed
\end{proof}

\begin{remark}
\label{bar b is independent of x 2}
The coefficient $\bar b$ is assumed to be independent of $x$ (Assumption \ref{deterministic_part_assumption_on_coeffi_space-time_white_noise}), and it is a sufficient condition for the existence of a global solution. Indeed, since a uniform $L_1$ bound of the solution is employed to extend existence time, we show that all the paths of $\|u(t,\cdot)\|_{L_1}$ are bounded. Intuitively, if we take integration and expectation to \eqref{burger's_eq_space_time_white_noise_super_linear}, the nonlinear term $\bar b |u|^\lambda u_x$ is removed since the nonlinear term is interpreted as 
\begin{equation}
\label{nonlinear_removed}
\bar b |u|^\lambda u_x = \bar b \frac{1}{1+\lambda} \left( |u|^{1+\lambda} \right)_{x} = \bar b \frac{1}{1+\lambda} \left( |u|^{\lambda}\cdot|u| \right)_{x}
\end{equation}
and $\bar b$ is independent of $x$. Thus, it turns out that $\|u(t,\cdot)\|_{L_1}$ is a local martingale, and its trajectories are bounded almost surely; see Lemma \ref{L_1_bound_small_lambda_STWN_lemma}.

\end{remark}

\begin{remark}
\label{super-linear_remark_1}
\begin{enumerate}[(i)]

\item
\label{super-linear_remark_1_1}
Note that $\lambda$ is assumed to be less than 1. Indeed, due to the \textit{unbounded} diffusion coefficient, we only obtain the uniform $L_1$ bound of the solution instead of $L_p$ bound ($p>1$). Since the surplus part $|u|^\lambda$ in \eqref{nonlinear_removed} should be summable for some $s = 1/\lambda > 1$, $\lambda$ is less than $1$.

\item 
Because $p > \frac{6}{1-2\kappa}$, there are $\alpha$ and $\beta$ satisfying \eqref{condition_for_alpha_beta}.

\item 
The condition on the initial data $u_0\in L_1(\Omega;L_1)$ is employed to obtain the uniform $L_1$ bound of the solution.
\end{enumerate}

\end{remark}

\begin{remark}

The regularity of the solution depends on $\lambda$ and $\lambda_0$. Indeed, the uniform $L_1$ bound of solutions is achieved, and the bound is employed to prove the local solution is non-explosive.
Then, restrictions on solution regularity are required to control the nonlinear terms $|u|^\lambda u_x$ and $|u|^{\lambda_0+1}$ since the surplus parts $|u|^\lambda$ and $|u|^{\lambda_0}$ should be summable to power $s = 1/\lambda$ and $2s_0 = 1/\lambda_0$, respectively; see Lemma \ref{Non_explosion_small_lambda_STWN}.

\end{remark}
\begin{remark}

The H\"older regularity of the solution introduced in Theorem \ref{theorem_burgers_eq_STWN_super_linear} depends on $\alpha$ and $\beta$. For example, for small $\ep > 0$, let $\alpha = \frac{1}{p}+\frac{\ep}{4}$ and $\beta = \frac{1}{p}+\frac{\ep}{2}$. Then, we have
\begin{equation*}
\sup_{t\leq T}|u(t,\cdot)|_{C^{\frac{1}{2}-\kappa-\frac{3}{p}-\ep}(\bR)} < \infty,
\end{equation*}
almost surely. Similarly, for small $\ep > 0$, set $\alpha = \frac{1}{2}\left(\frac{1}{2}-\kappa-\frac{1}{p}\right)-\ep$ and $\beta = \frac{1}{2}\left(\frac{1}{2}-\kappa-\frac{1}{p}\right)-\frac{\ep}{2}$. Then, we have
\begin{equation*}
\sup_{x\in\bR}|u(\cdot,x)|_{C^{\frac{1}{4}-\frac{\kappa}{2}-\frac{3}{2p}-\ep}([0,T])}  < \infty
\end{equation*}
almost surely.

\end{remark}

Next, we provide the uniqueness of the solution in $p$ to achieve the maximal H\"older regularity of the solution.

\begin{theorem}
\label{uniqueness_in_p_2}
Assume that all the conditions of Theorem \ref{theorem_burgers_eq_STWN_super_linear} holds. Let $u\in \cH_{p,loc}^{1/2-\kappa}$ be the solution of equation \eqref{burger's_eq_space_time_white_noise_super_linear} introduced in Theorem \ref{theorem_burgers_eq_STWN_super_linear}. If $q>p$ and $u_0\in U_{q}^{1/2-\kappa}\cap L_{1}(\Omega;L_{1})$, then $u\in \cH_{q,loc}^{1/2-\kappa}$.
\end{theorem}
\begin{proof}
See {\bf{Proof of Theorem \ref{uniqueness_in_p_2}}} in Section \ref{Proof of the second case}.
\qed
\end{proof}

By Theorems \ref{theorem_burgers_eq_STWN_super_linear} and \ref{uniqueness_in_p_2}, the maximal H\"older regularity of the solution follows.

\begin{corollary}
\label{maximal holder regularity2}
Suppose $u_0\in U_p^{1/2-(\lambda - 1/2)\vee \lambda_0}\cap L_1(\Omega;L_1)$ for all $p>2$. Then, for $T>0$ and small $\ep>0$, we have
\begin{equation}
\label{holder_regularity_large_lambda}
\sup_{t\leq T}|u(t,\cdot)|_{C^{1/2-(\lambda - 1/2)\vee \lambda_0-\ep}(\bR)} + \sup_{x\in\bR}|u(\cdot,x)|_{C^{\frac{1/2-(\lambda - 1/2)\vee \lambda_0}{2}-\ep}([0,T])}  < \infty
\end{equation}
almost surely.

\end{corollary}
\begin{proof}
Even though the proof is similar to the proof of Corollary \ref{maximal holder regularity1}, we contain the proof for the completeness of the article. Let $T\in(0,\infty)$ and $\kappa = \left(\lambda-\frac{1}{2}\right)\vee\lambda_0+\frac{\ep}{4}$ for small $\ep>0$. Note that for each $p > \frac{6}{1-2\kappa}$, there exists a unique solution $u = u_p\in \cH_p^{1/2-\kappa}(T)$ to equation \eqref{burger's_eq_space_time_white_noise_super_linear} by Theorem \ref{theorem_burgers_eq_STWN_super_linear}. Since $u_0 \in U_p^{1/2-(\lambda - 1/2)\vee\lambda_0}$ for all $p > 2$, Theorem \ref{uniqueness_in_p_2} implies that all the solutions $u = u_p$ coincides. Then, by considering $p = \frac{12}{\ep}$, $\alpha = \frac{1}{p} + \frac{\ep}{8}$, and $\beta = \frac{1}{p} + \frac{\ep}{4}$ in Corollary \ref{embedding_corollary}, almost surely 
\begin{equation*}
\sup_{t\leq T} |u(t,\cdot)|_{C^{1/2-(\lambda - 1/2)\vee \lambda_0-\ep}(\bR)} \leq |u|_{C^{\alpha-\frac{1}{p}}\left([0,T];C^{1/2-\kappa-2\beta-1/p}(\bR)\right)}<\infty.
\end{equation*}
On the other hand, if we consider $p = \frac{4}{\ep}$, $\alpha = \frac{1}{2}\left( \frac{1}{2}-\kappa-\frac{1}{p} \right) - \frac{\ep}{2}$, and $\beta = \frac{1}{2}\left( \frac{1}{2}-\kappa - \frac{1}{p} \right)- \frac{\ep}{4}$ in Corollary \ref{embedding_corollary}, almost surely 
\begin{equation*}
\sup_{x\in \bR} |u(\cdot,x)|_{C^{\frac{1}{2}\left( \frac{1}{2}-\kappa \right) - \ep}([0,T])}\leq |u|_{C^{\alpha-\frac{1}{p}}\left([0,T];C^{1-2\beta-\frac{1}{p}}(\bR)\right)}<\infty.
\end{equation*}
The corollary is proved.

\qed
\end{proof}


\section{Proof of the first case: generalized Burgers' equation with the bounded Lipschitz diffusion coefficient \texorpdfstring{$\sigma(u)$}{Lg}} 
\label{Proof of the first case}

This section suggests proof of Theorems \ref{theorem_burgers_eq_large_lambda_STWN} and \ref{uniqueness_in_p_1}. In the proof of Theorem \ref{theorem_burgers_eq_large_lambda_STWN}, the uniqueness of a global solution follows from the uniqueness of the local ones; see Lemma \ref{cut_off_lemma_large_lambda_STWN}. Also, we construct a global solution candidate by pasting local ones to show the existence of a global solution; see Remark \ref{def_u}. To check the global solution candidate $u$ is non-explosive, we show that for any $T<\infty$,
\begin{equation}
\label{non_explosive}
P\left( \left\{ \omega\in\Omega:\sup_{t\leq T,x\in\bR} |u(t,x)| > R \right\} \right) \to 0
\end{equation}
as $R\to\infty$. Then, to prove \eqref{non_explosive}, we separate the local solution $u_m$ in two parts; noise-related part $v$ and nonlinear-related part $w_m:=u_m-v$. In Lemma \ref{noise_cancelling_lemma}, we control the noise-related part $v$ with the initial data $u_0$ and bound of $\sigma(u)$. On the other hand, in Lemma \ref{Lq_bound_2}, the other part $w_m$ is dominated by $v$. In the proof of Lemma \ref{Lq_bound_2}, we employ the fundamental theorem of calculus with $p-1+\lambda\in 2\bN$.

\vspace{1mm}

First of all, we introduce a theorem describing the uniqueness, existence, and $L_p$-regularity of a solution to semilinear equation 
\begin{equation} \label{nonlinear_equation}
du = \left(au_{xx} + b u_{x} + cu + f(u)\right) dt + g^k(u) dw_t^k,\quad t>0\,; \quad u(0,\cdot) = u_0(\cdot),
\end{equation}
where $f(u)$ and $g(u)$ satisfy Assumption \ref{assumption_on_f_and_g}. This result is used to construct local solutions.

\begin{assumption}[$\tau$] \label{assumption_on_f_and_g}
\begin{enumerate}[(i)]
\item 
The function $f(t,x,u)$ is $\cP\times\cB(\bR)\times\cB(\bR)$-measurable and $f(t,x,0)\in \bH_{p}^{\gamma}(\tau)$.

\item 
The function $g^k(t,x,u)$ is $\cP\times\cB(\bR)\times\cB(\bR)$-measurable and $g(t,x,0) = \left( g^1(t,x,0),g^2(t,x,0),\dots \right)\in \bH_{p}^{\gamma+1}(\tau,\ell_2)$.

\item 
For any $\ep>0$, there exists a constant $N_\ep$ such that
\begin{equation*}
\begin{aligned}
\| f(u) - f(v) \|_{\bH_{p}^\gamma(\tau)} + \| g(u) &- g(v) \|_{\bH_{p}^{\gamma+1}(\tau,\ell_2)} \\
&\leq \ep\| u-v \|_{\bH_p^{\gamma+2}(\tau)} + N_\ep \| u-v \|_{\bH_p^{\gamma+1}(\tau)}
\end{aligned}
\end{equation*}
for any $u,v\in \bH_p^{\gamma+2}(\tau)$.

\end{enumerate}
\end{assumption}

\begin{theorem} 
\label{theorem_nonlinear_case}
Let $\tau\leq T$ be a bounded stopping time, $\gamma\in[-2,-1]$, and $p\geq2$.  Suppose Assumptions \ref{deterministic_part_assumption_on_coeffi_space-time_white_noise} and \ref{assumption_on_f_and_g} ($\tau$) hold. Then, for any $u_0\in U_p^{\gamma+2}$, equation \eqref{nonlinear_equation} has a unique solution $u\in\cH_p^{\gamma+2}(\tau)$ such that
\begin{equation} 
\label{nonlinear_estimate}
\| u \|_{\cH_p^{\gamma+2}(\tau)} \leq N\left(\|f(0)\|_{\bH_p^{\gamma}(\tau)} + \|g(0)\|_{\bH_p^{\gamma+1}(\tau,\ell_2)} + \| u_0 \|_{U_p^{\gamma+2}}\right),
\end{equation}
where $N$ depends on constants $d,p,\gamma,K,T$ and the function $N_\ep$ $(\ep>0)$. 
\end{theorem}
\begin{proof}
We only consider case $\tau \leq T$ because case $\tau = T$ follows from \cite[Theorem 5.1]{kry1999analytic}. 

\textit{Step 1. (Existence). } Set
$$
\bar{f}(t, u) := 1_{t\leq \tau} f(t,u)\quad\text{and}\quad \bar{g}(t, u) := 1_{t\leq \tau} g(t,u).
$$
Note that $\bar{f}(u)$ and $\bar{g}(u)$ satisfy Assumption \ref{assumption_on_f_and_g} $(T)$. Therefore, by \cite[Theorem 5.1]{kry1999analytic}, there exists a unique solution $u\in \cH_{p}^{\gamma+2}(T)$ such that $u$ satisfies equation \eqref{nonlinear_equation} with $\bar{f}, \bar{g}$, instead of $f,g$. Since $\tau\leq T$, we have $u\in\cH_{p}^{\gamma+2}(\tau)$ and $u$ satisfies equation \eqref{nonlinear_equation} and estimate \eqref{nonlinear_estimate} with $f,g$.

\textit{Step 2. (Uniqueness). } Suppose $v\in \cH_{p}^{\gamma+2}(\tau)$ is another solution to equation \eqref{nonlinear_equation}. Then, \cite[Theorem 5.1]{kry1999analytic} yields that there exists a unique solution $\bar{v}\in \cH_{p}^{\gamma+2}(T)$ satisfying
\begin{equation}
\label{equation_in_proof_of_uniqueness}
d\bar{v}=\left(a^{ij}\bar{v}_{x^ix^j} + b^i\bar{v}_{x^i} + c\bar{v} + \bar{f}(v) \right)dt + \sum_{k=1}^{\infty} \bar{g}^k(v)  dw^k_t, \quad 0<t<T
\end{equation}
with the initial data $\bar{v}(0,\cdot)=u_0$.
Notice that $\bar{f}(v)$ and $\bar{g}(v)$ are used in \eqref{equation_in_proof_of_uniqueness}, instead of $\bar{f}(\bar{v})$ and $\bar{g}(\bar{v})$.  Let  $\tilde{v}:=v-\bar{v}$. Then, for each $\omega\in\Omega$,
$$
d\tilde{v}=\left(a^{ij}\tilde{v}_{x^ix^j} + b^i\tilde{v}_{x^i} + c\tilde{v} \right)dt, \quad 0<t<\tau\, ; \quad \tilde{v}(0,\cdot)=0.
$$
Thus, \cite[Theorem 5.1]{kry1999analytic} implies $v(t,\cdot) = \bar{v}(t,\cdot)$ in $H_p^{\gamma+2}$ for all $t\leq \tau$ almost surely. Therefore, we can replace $\bar{f}(v), \bar{g}(v)$ with $\bar{f}(\bar{v}),\bar{g}(\bar{v})$ in equation \eqref{equation_in_proof_of_uniqueness}. Similarly, there exists $\bar{u}\in \cH_p^{\gamma+2}(T)$ satisfying \eqref{nonlinear_equation} and $u(t,\cdot) = \bar{u}(t,\cdot)$ in $H_p^{\gamma+2}$ for all $t\leq \tau$ almost surely. Since the uniqueness result in $\cH_{p}^{\gamma+2}(T)$ yields $\bar{u}=\bar{v}$ in $\cH_{p}^{\gamma+2}(T)$, we have $u = v$ in $H_p^{\gamma+2}$ for almost every $(\omega,t)\in\opar0,\tau\cbrk$. Thus, we have $u = v$ in $\cH_p^{\gamma+2}(\tau)$. The theorem is proved.
\qed
\end{proof}

Let $h(\cdot)\in C_c^\infty(\bR)$ be a nonnegative function such that
$h(z) = 1$ on $|z|\leq 1$ and $h(z) = 0$ on $|z|\geq2$. Then, for $m\in\bN$, define
\begin{equation}
\label{cut-off function}
h_m(z) := h(z/m).
\end{equation}

Below lemma provides the existence and uniqueness of local solutions $u_m$.

\begin{lemma} \label{cut_off_lemma_large_lambda_STWN}
Let $\lambda\in(0,\infty)$ $T\in(0,\infty)$, $\kappa\in (0,1/2)$, and $p \geq 2$. Suppose Assumptions \ref{deterministic_part_assumption_on_coeffi_space-time_white_noise} and \ref{stochastic_part_assumption_on_coeffi_space_time_white_noise_large_lambda} hold. For a bounded stopping time $\tau\leq T$, $m\in\bN$, and nonnegative initial data $u_0\in U_{p}^{1/2-\kappa}$, there exists $u_m \in \cH_p^{1/2-\kappa}(\tau)$  satisfying equation 
\begin{equation} 
\label{cut_off_equation_large_lambda_STWN}
du = \left(au_{xx} + b u_{x} + cu +  \bar{b}\left(u_+^{1+\lambda} h_m(u)\right)_{x}  \right) dt + \sigma(u)\eta_k dw_t^k; \quad u(0,\cdot) = u_0(\cdot),
\end{equation}
where $(t,x)\in(0,\tau)\times\bR$. Furthermore, $u_m \geq 0$. 
\end{lemma}
\begin{proof}
Notice that for $u,v\in \bR$, we have
\begin{equation}
\label{nonlinear_cutoff}
\left| u_+^{1+\lambda}h_m(u) - v_+^{1+\lambda}h_m(v) \right| \leq N_m|u-v|.
\end{equation}
Indeed,
\begin{equation} \label{lipschitz_check_cut_off}
\begin{aligned}
&\left| u_+^{1+\lambda}h_m(u) - v_+^{1+\lambda}h_m(v) \right| \\
&=
\begin{cases}
\,\,\left| u^{1+\lambda}h_m(u) - v^{1+\lambda}h_m(v) \right| \leq N_m|u-v|  &\text{ if}\quad u,v\geq0,\\
\,\,u^{1+\lambda}h_m(u) \leq (2m)^\lambda u \leq (2m)^\lambda (u-v)\leq N_m|u-v| &\text{ if}\quad u\geq0, v < 0,\\
\,\,v^{1+\lambda}h_m(v) \leq (2m)^\lambda v \leq (2m)^\lambda (v-u)\leq N_m|u-v| &\text{ if}\quad u<0, v\geq0,\\
\,\,0\leq N_m|u-v| &\text{ if}\quad u,v < 0.
\end{cases}
\end{aligned}
\end{equation}
Thus, Lemma \ref{prop_of_bessel_space} \eqref{bounded_operator}, Remark \ref{Kernel}, and Minkowski's inequality implies that for $u,v\in \bH_p^{1/2-\kappa}(\tau)$,
\begin{equation}
\begin{aligned}
&\left\| \left((u_+(t,\cdot))^{1+\lambda} h_m(u(t,\cdot))\right)_{x} - \left((v_+(t,\cdot))^{1+\lambda} h_m(v(t,\cdot))\right)_{x} \right\|^p_{H_p^{-3/2-\kappa}} \\
&\quad \leq N\left\| (u_+(t,\cdot))^{1+\lambda} h_m(u(t,\cdot)) - (v_+(t,\cdot))^{1+\lambda} h_m(v(t,\cdot)) \right\|^p_{H^{-1/2-\kappa}_p} \\
&\quad\leq N\int_{\bR} \left( \int_{\bR} \left| R_{1/2+\kappa}(y) \right| \left| \left(u_+^{1+\lambda}h_m(u) - v_+^{1+\lambda}h_m(v)\right)(t,x-y) \right| dy\right)^p dx \\
&\quad \leq N_m\left(\int_\bR \left| R_{1/2+\kappa}(y) \right| dy\right)^{p}\int_\bR |u(t,x) - v(t,x)|^p dx \\
\end{aligned}
\end{equation}
and
\begin{equation} \label{l_2_computation}
\begin{aligned}
&\left\| \sigma(u)\eta - \sigma(v)\eta \right\|_{H_p^{-1/2-\kappa}(\ell_2)}^p \\
& \leq \int_\bR  \left( \sum_k \left( \int_\bR \left| R_{1/2+\kappa}(x-y) \right|((\sigma(u) - \sigma(v))(s,y)\eta_k(y) dy \right)^2 \right)^{\frac{p}{2}} dx \\
& \leq \int_\bR  \left( \int_\bR \left|R_{1/2+\kappa}(x-y)\right|^2(\sigma(s,y,u(s,y)) - \sigma(s,y,v(s,y)))^2 dy \right)^{p/2} dx \\
& \leq \int_\bR  \left( \int_\bR \left|R_{1/2+\kappa}(y)\right|^2(u(s,x-y) - v(s,x-y))^2 dy \right)^{p/2} dx \\
& \leq \left( \int_\bR \left| R_{1/2+\kappa}(y) \right|^2 dy \right)^{p/2}\int_\bR  |u(s,x) - v(s,x)|^p  dx \\
\end{aligned}
\end{equation}
on $(\omega,t)\in\opar0,\tau\cbrk$. Since Remark \ref{Kernel} implies 
$$\int_\bR \left|R_{1/2+\kappa}(y)\right| dy + \int_\bR \left| R_{1/2+\kappa}(y) \right|^2 dy<\infty,
$$ 
by taking integration with respect to $(\omega,t)$ and applying Lemma \ref{prop_of_bessel_space} \eqref{multi_ineq}, we have
\begin{equation*}
\begin{aligned}
&\left\| \left(u_+^{1+\lambda} h_m(u)\right)_{x} - \left(v_+^{1+\lambda} h_m(v)\right)_{x} \right\|^p_{\bH_p^{-3/2-\kappa}(\tau)} + \left\| \sigma(u)\eta - \sigma(v)\eta \right\|_{\bH_p^{-1/2-\kappa}(\tau,\ell_2)}^p \\
&\quad\leq  N_m\|u-v\|_{\bL_p(\tau)}^p \\
&\quad\leq  \ep\|u-v\|_{\bH_p^{1/2-\kappa}(\tau)}^p + N_\ep\|u-v\|_{\bH_p^{-3/2-\kappa}(\tau)}^p
\end{aligned}
\end{equation*}
for any $\ep>0$. Therefore, Theorem \ref{theorem_nonlinear_case} yields that there exists $u_m\in \cH_p^{1/2-\kappa}(\tau)$ satisfying equation \eqref{cut_off_equation_large_lambda_STWN}.

To prove $u_m\geq0$, take a sequence of functions $\{u^n_0 \in  U_p^{1}:u^n_0\geq 0,n\in\bN\}$ such that $u^n_0 \to u_0$ in $U^{1/2-\kappa}_{p}$. By Theorem \ref{theorem_nonlinear_case}, there exists a unique solution  $u^n_m\in \cH_p^1(\tau)$ satisfying equation
\begin{equation*}
du^n_m = \left(au_{mxx}^n + b u_{mx}^n + cu_m^n +  \bar{b}\left((u^n_{m+})^{1+\lambda} h_m(u_m^n)\right)_{x}  \right) dt + \sum_{k = 1}^n\sigma(u_m^n)\eta_k dw_t^k
\end{equation*}
on $(t,x)\in(0,\tau)\times\bR$ with initial data $u_m^n(0,\cdot) = u^n_0(\cdot)$.
Again, by Theorem \ref{theorem_nonlinear_case},
\begin{equation*}
\begin{aligned}
&\|u_m - u_m^n\|_{\cH_p^{1/2-\kappa}(\tau\wedge t)}^p - N\|u_0 - u_0^n\|_{U_p^{1/2-\kappa}}^p\\
&\quad \leq N\left\| \bar{b}\left((u_{m+})^{1+\lambda} h_m(u_m)\right)_{x} - \bar{b}\left((u_{m+}^n)^{1+\lambda} h_m(u_{m}^n)\right)_{x} \right\|_{\bH_p^{-3/2-\kappa}(\tau\wedge t)}^p \\
&\quad\quad+ N\left\| \sigma(u_{m})\eta - \sigma(u_{m}^n)\eta1_{k\leq n} \right\|_{\bH^{-1/2-\kappa}_p(\tau\wedge t,\ell_2)}^p \\
&\quad\leq N_m\left\| u_m - u_m^n \right\|_{\bH_p^{-1/2-\kappa}(\tau\wedge t)}^p + N\left\| \left(\sigma(u_{m}) - \sigma(u_{m}^n)\right)\eta1_{k\leq n} \right\|_{\bH^{-1/2-\kappa}_p(\tau\wedge t,\ell_2)}^p \\
&\quad\quad+ N\left\| \sigma(u_{m})\eta1_{k > n} \right\|_{\bH^{-1/2-\kappa}_p(\tau\wedge t,\ell_2)}^p \\
&\quad\leq N_m \| u_{m} - u_{m}^n \|_{\bL_p(\tau\wedge t)}^p  + N\left\| \sigma(u_{m})\eta1_{k > n} \right\|_{\bH^{-1/2-\kappa}_p(\tau,\ell_2)}^p,
\end{aligned}
\end{equation*}
where $N = N(m,d,p,\kappa,K,T)$. Then, Theorem \ref{embedding} \eqref{gronwall_type_ineq} implies
\begin{equation*}
\|u_m - u_m^n\|_{\cH_p^{1/2-\kappa}(\tau\wedge t)}^p \leq N \|u_0 - u_0^n\|_{U_p^{1/2-\kappa}}^p + N\left\| \sigma(u_{m})\eta 1_{k > n} \right\|_{\bH^{-1/2-\kappa}_p(\tau,\ell_2)}^p,
\end{equation*}
where $N = N(m,d,p,\kappa,K,T)$. Note that
\begin{equation*}
\begin{aligned}
&\| \sigma(u_m)\eta 1_{k>n} \|^p_{H_{p}^{-1/2-\kappa}(\ell_2)}  \\
&\quad \leq N  \int_{\bR} \left( \sum_{k>n} \left[ \int_{\bR}R_{1/2+\kappa}(x-y) \sigma(t,y,u_m(t,y)) \eta_k(y) dy \right]^2 \right)^{p/2} dx.
\end{aligned}
\end{equation*}
Also,
\begin{equation*}
\begin{aligned}
&\sum_{k>n} \left[ \int_{\bR}R_{1/2+\kappa}(x-y) \sigma(t,y,u_m(t,y)) \eta_k(y) dy \right]^2 \\
&\quad\leq \sum_{k} \left[ \int_{\bR}R_{1/2+\kappa}(x-y) \sigma(t,y,u_m(t,y)) \eta_k(y) dy \right]^2 \\
&\quad\leq  \int_{\bR}R^2_{1/2+\kappa}(x-y) | \sigma(t,y,u_m(t,y))|^2  dy  \\
&\quad\leq  K^2\int_{\bR}R^2_{1/2+\kappa}(y) | u_m(t,x-y)|^2  dy.
\end{aligned}
\end{equation*}
Thus, by the dominated convergence theorem, we have
\begin{equation*}
\left\| \sigma(t,\cdot,(u_m(t,\cdot)) \eta 1_{k>n} ) \right\|^p_{H_{p}^{-1/2-\kappa}(\ell_2)} \to 0 \quad \text{as}\quad n\to \infty
\end{equation*}
for almost all $(\omega,t)$. Then, by using the dominated convergence theorem again, we have
$$\|  \sigma(t,\cdot,(u_m(t,\cdot)) \eta 1_{k>n} ) \|^p_{\mathbb{H}_{p}^{-1/2-\kappa}(\tau,\ell_2)} \to 0 \quad \text{as}\quad n\to \infty.
$$ 
Thus, it suffices to show that $u^n_{m}\geq0$. Since $u^n_{m}\in\cH_p^1(\tau)$,  By applying \cite[Theorem 2.5]{krylov2007maximum} with $f = \bar b\left((u^n_{m+})^{1+\lambda}h_m(u^n_m)\right)_{x}$, we have $u^n_{m}\geq0$. 
The lemma is proved.

\qed
\end{proof}

The following lemma provides $L_1$ bound of local solution $u_m$ on $\Omega\times(0,T)\times \bR$. This bound is used to control the noise-related part of the local solution; see \eqref{noise_cancelling_estimate}. Recall that $h(\cdot)\in C_c^\infty(\bR)$ is a nonnegative function satisfying
$h(z) = 1$ on $|z|\leq 1$ and $h(z) = 0$ on $|z|\geq2$. Also, $h_m$ is defined as
$$ h_m(z) := h(z/m).
$$

\begin{lemma}
\label{L1_bound}
Suppose all the assumptions of Lemma \ref{cut_off_lemma_large_lambda_STWN} hold for $\tau = T$ and we assume that $u_0\in U_{p}^{1/2-\kappa}\cap L_1(\Omega;L_1)$. Let $u_m\in \cH_p^{1/2-\kappa}(T)$ be the solution to equation \eqref{cut_off_equation_large_lambda_STWN} introduced in Lemma \ref{cut_off_lemma_large_lambda_STWN}. Then, we have
\begin{equation} \label{L_1_bound_STWN_1}
\bE\int_0^T \| u_m(t,\cdot) \|_{L_1}dt \leq N \| u_0 \|_{L_1(\Omega;L_1)}
\end{equation}
where $N = N(K,T)$.
\end{lemma}
\begin{proof}

Note that $h_k(x) = h(x/k) \in C_c^\infty(\bR)$ and $h_k(x)\to1$ as $k\to \infty$. For a bounded stopping time $\tau \leq T$ and $k\in\bN$, by employing It\^o's formula, expectation, and  integration by parts, we have
\begin{equation}
\label{L1_bound_calculation_1}
\begin{aligned}
&\bE e^{-4K\tau}\left(u_m(\tau ,\cdot),h_k  \right)\\
& = \bE(u_0, h_k) + k^{-2}\bE\int_0^\tau\int_{\bR} u_m(s,x)a(s,x)h_{xx}(x/k)e^{-4Ks}dxds \\
& + k^{-1}\bE\int_0^\tau\int_{\bR} u_m(s,x)\left( 2a_{x} - b - \bar b (u_{m+})^{\lambda} h_m(u_m)  \right)(s,x) h_x(x/k)e^{-4Ks}dxds \\
& + \bE\int_0^\tau\int_{\bR} u_m(s,x) \left( a_{xx}(s,x) - b_x(s,x) + c(s,x) - 4K \right) h(x/k) e^{-4Ks}dxds.
\end{aligned}
\end{equation}
Since $u_0\in L_1(\Omega;L_1)$ and \eqref{boundedness_of_deterministic_coefficients_space-time_white_noise}, we have
\begin{equation}
\label{L1_bound_calculation_2}
\begin{aligned}
&e^{-4KT}\bE(u_m(\tau,\cdot),h_k) \\
&\quad\leq \bE\int_\bR u_0(x)dx + k^{-2}N\bE\int_0^\tau \int_{\bR} |u_m(s,x)| |h_{xx}(x/k)|  dxds \\
&\quad\quad\quad + k^{-1}N_m\bE\int_0^\tau \int_{\bR} |u_m(s,x)| |h_x(x/k)| dxds.
\end{aligned}
\end{equation}
Let $q := \frac{p}{p-1}$. By H\"older's inequality and changing of variable, we have
\begin{equation*}
\begin{aligned}
&k^{-2}\bE\int_0^\tau\int_{\bR} |u_m(s,x)| |h_{xx}(x/k)| dxds \\
&\quad \leq k^{-2}\| u_m \|_{\bL_p(\tau)} \left(\int_{\bR} |h_{xx}(x/k)|^q dx\right)^{1/q} \\
&\quad \leq k^{-2+1/q}N_m
\end{aligned}
\end{equation*}
and
\begin{equation*}
\begin{aligned}
k^{-1}\bE\int_0^\tau\int_{\bR} |u_m(s,x)||h_{x}(x/k)| dxds &\leq k^{-1}\bE\int_0^\tau\int_{\bR} |u_m(s,x)||h_x(x/k)| dxds\\
&\leq k^{-1}\| u_m \|_{\bL_p(\tau)} \| h_x(\cdot/k) \|_{L_q} \\
&\leq k^{-1+1/q}N_m \| h_x \|_{L_{q}}\\
&\leq k^{-1+1/q}N_m.
\end{aligned}
\end{equation*}
Therefore, we have
\begin{equation} \label{two kinds of L1bound}
\bE(u_m(\tau,\cdot),h_k) \leq \bE\| u_0 \|_{L_1(\Omega\times\bR)} e^{4KT}+ k^{-1+1/q}  N_me^{4KT},
\end{equation}
where $N_M$ is independent of $k$. By letting $\tau = t$ and taking integration on $t\in(0,T)$, we have
\begin{equation*}
\bE\int_0^T(u_m(t,\cdot),h_k)dt \leq T\bE\int_\bR u_0(x)dx e^{4KT}+ k^{-1+1/q}N_mT e^{4KT},
\end{equation*}
where $N_{m}$ is independent of $k$. By letting $k\to\infty$, we have
\begin{equation*}
\bE \int_0^T \| u_m(s,\cdot) \|_{L_1}ds \leq N(K,T)\bE\| u_0 \|_{L_1}.
\end{equation*}
The lemma is proved.
\qed
\end{proof}

\begin{remark} \label{def_u}
Now we introduce a global solution candidate. Let $T\in\bN$, $\kappa\in(0,1/2)$, $p > \frac{6}{1-2\kappa}$ and $m\in\bN$. Suppose $u_m\in \cH_p^{1/2-\kappa}(T)$ is the solution introduced in Lemma \ref{cut_off_lemma_large_lambda_STWN}. By Corollary \ref{embedding_corollary}, we have $u_m\in C([0,T];C(\bR))$ (a.s.).
Thus, for $R\in\{ 1,2,\dots,m \}$, we can set $\tau_m^R$
\begin{equation} 
\label{stopping_time_taumm}
\tau_m^R := \inf\left\{  t\in[0,T] : \sup_{x\in\bR} |u_m(t,x)|\geq R \right\}.
\end{equation}
Note that 
\begin{equation}
\label{local_existence_time_increasing}
\tau_R^R \leq \tau_m^m.
\end{equation}
Indeed, if $R = m$, \eqref{local_existence_time_increasing} is obvious. If $R<m$, we have $u_m\wedge m=u_m\wedge m\wedge R  = u_m \wedge R$ for $t\leq\tau_m^R$. Therefore, $u_m$ and $u_R$ satisfy
\begin{equation*}
du = \left(au_{xx} + b u_{x} + cu +  \bar{b}\left(u^{1+\lambda}h_R(u) \right)_{x}  \right) dt + \sigma(u)\eta_k dw_t^k, \quad0<t\leq\tau_m^R
\end{equation*}
with the initial data $u(0,\cdot)=u_0$.
On the other hand, $u_R\wedge R = u_R\wedge R\wedge m = u_R\wedge m$ for $t\leq \tau_R^R$. Thus, $u_m$ and $u_R$ satisfy
\begin{equation*}
du = \left(au_{xx} + b u_{x} + cu +  \bar{b}\left(u^{1+\lambda}h_m(u) \right)_{x}  \right) dt + \sigma(u)\eta_k dw_t^k, \quad0<t\leq\tau_R^R
\end{equation*}
with the initial data $u(0,\cdot)=u_0.$
Notice that the uniqueness result in Lemma \ref{cut_off_lemma_large_lambda_STWN} yields that $u_m=u_R$ for all $t \leq (\tau_m^R\vee\tau_R^R) \wedge T$, for any positive integer $T$. Also, note that $\tau_R^R=\tau_m^R$ (a.s.).  Indeed, for $t<\tau_m^R$, 
\begin{equation*}
\sup_{s\leq t}\sup_{x\in\bR}|u_R(s,x)|=\sup_{s\leq t}\sup_{x\in\bR}|u_m(s,x)|\leq R,
\end{equation*}
and this implies  $\tau_m^R\leq\tau_R^R$. Similarly, $\tau_m^R\geq\tau_R^R$. Besides, we have $\tau_m^R \leq \tau_m^m$ since $m > R$. Therefore, $ \tau_R^R \leq \tau_m^m$. 

Define
\begin{equation*}
u := u_{m}1_{[0,\tau_m^m\wedge T]}(t)
\end{equation*}
and set $\tau_\infty := \limsup_{m\to\infty}\limsup_{T\to\infty}\tau_m^m\wedge T$. It should be remarked that $u(t,x)$ is well-defined on $\Omega\times[0,\infty)\times\bR$ and the nontrivial domain of $u$ is $\Omega\times[0,\tau_\infty)\times\bR$.

\end{remark}

The following lemma introduces a noise dominating part of the solution and the estimate of it.

\begin{lemma} \label{noise_cancelling_lemma}
Let $T\in(0,\infty)$, $\kappa\in(0,1/2)$, and $p>\frac{6}{1-2\kappa}$. If $u_0\in U_p^{1/2-\kappa}\cap L_1(\Omega;L_1)$, there exists $v \in \cH_p^{1/2-\kappa}(T)$ such that 
\begin{equation} \label{noise_cancelling_eq}
dv = \left(av_{xx} + b v_{x} + cv \right) dt + \sigma(u)\eta_k dw_t^k,\quad (t,x)\in(0,T)\times\bR\,; \quad v(0,\cdot) = u_0.
\end{equation}
Furthermore, we have
\begin{equation} \label{noise_cancelling_estimate}
\bE\sup_{t\leq T,x\in\bR}|v(t,x)|^p \leq N(p,\kappa,K,T)\left(\|u_0\|_{U_p^{1/2-\kappa}}^p + \bE\| u_0 \|_{L_1}\right)
\end{equation}
and
\begin{equation} \label{noise_cancelling_estimate_2}
\bE\sup_{t\leq T}\|v(t,\cdot)\|^{1/2}_{L_1} \leq  N(K,T)\bE\| u_0 \|^{1/2}_{L_1}.
\end{equation}
\end{lemma}
\begin{proof}
Let $T\in(0,\infty)$. First, we show that there exists $v\in\cH_p^{1/2-\kappa}(T)$ satisfying \eqref{noise_cancelling_eq} and \eqref{noise_cancelling_estimate}. To employ Theorem \ref{theorem_nonlinear_case}, we need to show
\begin{equation*}
\left\| \sigma(u)\eta  \right\|_{\bH_p^{-1/2-\kappa}(T,\ell_2)}^p <\infty.
\end{equation*}
Set $\tau_m^m$ as in \eqref{stopping_time_taumm}. Then, Remark \ref{Kernel}, Remark \ref{def_u}, and Fatou's lemma yield
\begin{equation} \label{noise_estimate_STWN_1}
\begin{aligned}
&\left\| \sigma(u)\eta  \right\|_{\bH_p^{-1/2-\kappa}(T,\ell_2)}^p \\
& = \bE\int_0^T\left\| |(1-\Delta)^{1/2+\kappa}\left(\sigma(s,u(s))\eta\right)(\cdot)|_{\ell_2}  \right\|_{L_p}^p  ds\\
& \leq \bE\int_0^{\tau_\infty\wedge T}\int_\bR  \left( \sum_k \left( \int_\bR R_{1/2+\kappa}(x-y)\sigma(s,y,u(s,y)) \eta_k(y) dy \right)^2 \right)^{p/2} dxds \\
& = \bE\int_0^{\tau_\infty\wedge T}\int_\bR  \left( \int_\bR \left|R_{1/2+\kappa}(x-y)\right|^2(\sigma(s,y,u(s,y)) )^2 dy \right)^{p/2} dxds \\
& \leq \liminf_{m\to\infty}\bE\int_0^{\tau_m^m\wedge T}\int_\bR  \left( \int_\bR \left|R_{1/2+\kappa}(x-y)\right|^2(\sigma(s,y,u_m(s,y)) )^2 dy \right)^{p/2} dxds \\
\end{aligned}
\end{equation}
Also, by Remark \ref{Kernel}, \eqref{boundedness_of_stochastic_coefficients_large_lambda_STWN}, Minkowski's inequality, and \eqref{L_1_bound_STWN_1}, we have
\begin{equation} \label{noise_estimate_STWN_2}
\begin{aligned}
&\bE\int_0^{\tau_m^m\wedge T}\int_\bR  \left( \int_\bR \left|R_{1/2+\kappa}(x-y)\right|^2(\sigma(s,y,u_m(s,y)) )^2 dy \right)^{p/2} dxds \\
& \leq K^{p} \bE\int_0^{\tau_m^m\wedge T} \int_\bR  \left( \int_\bR \left|R_{1/2+\kappa}(y)\right|^2(u_m(s,x-y) )^{2/p} dy \right)^{p/2} dxds \\
& \leq K^{p}\left( \int_\bR \left|R_{1/2+\kappa}(y)\right|^2dy \right)^{p/2}\bE\int_0^{\tau_m^m\wedge T} \int_\bR  |u_m(s,x)|  dxds \\
& \leq  N\| u_0 \|_{L_1(\Omega;L_1)},
\end{aligned}
\end{equation}
where $N = N(p,\kappa,K,T)$. Since $N$ is independent of $m$, \eqref{noise_estimate_STWN_1} and \eqref{noise_estimate_STWN_2} imply
\begin{equation}
\label{noise_estimate_STWN_3}
\left\| \sigma(u)\eta  \right\|_{\bH_p^{-1/2-\kappa}(T,\ell_2)}^p  \leq  N\| u_0 \|_{L_1(\Omega;L_1)},
\end{equation}
where $N = N(p,\kappa,K,T)$.
Therefore, by Theorem \ref{theorem_nonlinear_case}, there exists $v\in \cH_p^{1/2-\kappa}(T)$ satisfying \eqref{noise_cancelling_eq} and 
\begin{equation*}
\| v \|_{\cH_p^{1/2-\kappa}(T)}^p \leq
 N\|u_0\|_{U_p^{1/2-\kappa}}^p + 
 N\| \sigma(u)\eta \|_{\bH_p^{-1/2-\kappa}(T,\ell_2)}^p,
\end{equation*}
where $N = N(p,\kappa,K,T)$.
Thus, by \eqref{noise_estimate_STWN_3} and Theorem \ref{embedding} \eqref{large-p-embedding}, we have \eqref{noise_cancelling_estimate}. 

To show \eqref{noise_cancelling_estimate_2}, follow the proof of Lemma \ref{L1_bound}. Then, instead of \eqref{two kinds of L1bound}, for any bounded stopping time $\tau \leq T$,
\begin{equation*}
\bE (v(\tau,\cdot),h_k) \leq T\bE\int_\bR u_0(x)dx e^{4KT}+ k^{-1+1/q}N e^{4KT},
\end{equation*}
where $N = N(m,p,K,T)$. Then, (e.g. \cite[Theorem III.6.8]{diffusion})
\begin{equation*}
\bE\sup_{t\leq T}\| v(t,\cdot) \|_{L_1}^{1/2} \leq N(K,T)\bE\| u_0 \|_{L_1}^{1/2}.
\end{equation*}
The lemma is proved.

\qed
\end{proof}

The following lemma provides an estimate of the nonlinear part of the local solution; $w_m := u_m-v$.

\begin{lemma} \label{Lq_bound_2}
Suppose all the assumptions on Theorem \ref{theorem_burgers_eq_large_lambda_STWN} hold. Assume that $u_m\in\cH_p^{1/2-\kappa}(T)$ and $v\in\cH_p^{1/2-\kappa}(T)$ are solutions introduced in Lemmas \ref{cut_off_lemma_large_lambda_STWN} and \ref{noise_cancelling_lemma}, respectively. Let $\tau_m^m$ be a stopping time defined as in \eqref{stopping_time_taumm}. Then, for any stopping time $\tau\leq \tau_m^m$,
\begin{equation} 
\label{Lq_bound_2_estimate}
\begin{aligned}
&\int_\bR |u_m(\tau\wedge t,x) - v(\tau\wedge t,x)|^p dx \\
&\leq N(p,K,T)\left(\sup_{t\leq \tau,x\in\bR}|v(t,x)|^{\frac{p+\lambda-2}{2-\lambda}} + \sup_{t\leq \tau,x\in\bR}|v(t,x)|^{\frac{p(1+\lambda)-2}{2}}\right)\sup_{t\leq \tau} \| v(t,\cdot) \|_{L_1} 
\end{aligned}
\end{equation}
for all $t\leq \tau$ almost surely.
\end{lemma}
\begin{proof}
Let $\tau\leq \tau_m^m$ and set $w_m := u_m - v$. 
Then, for any $\phi\in \cS$, 
\begin{equation}
\label{the_other_part}
\begin{aligned}
(w_m(t,\cdot),\phi) &= \int_0^t (w_m(s,\cdot),(a(s,\cdot)\phi)_{xx}-(b(s,\cdot)\phi)_x+c(s,\cdot)\phi)ds\\
&\quad\quad-\int_0^t\bar b (s)\int_\bR (w_m(s,x)+v(s,x))^{1+\lambda}\phi_x dxds
\end{aligned}
\end{equation}
for all $t\leq \tau$ almost surely. Choose a nonnegative function $\zeta\in C_c^\infty(\bR)$ such that $\int_\bR \zeta(x)dx = 1$. For $x\in\bR$ and $\ep>0$, by plugging in $\phi(\cdot) := \frac{1}{\ep}\zeta\left(\frac{x-\cdot}{\ep}\right)$ to \eqref{the_other_part}, we have
\begin{equation*}
\begin{aligned}
w_m^{(\ep)}(t,x) &= \int_0^t  (aw_{m})_{xx}^{(\ep)}(s,x) -2(a_x w_m)^{(\ep)}_x(s,x) + (a_{xx}w_m)^{(\ep)}(s,x) ds \\
&\quad\quad-\int_0^t  (bw_m)_{x}^{(\ep)}(s,x) + (b_xw_m)^{(\ep)}(s,x) ds \\
&\quad\quad+\int_0^t (cw_m)^{(\ep)}(s,x) + \bar{b}(s) \left( (w_m(s,\cdot) + v(s,\cdot))^{1+\lambda} \right)^{(\ep)}_x(x)  ds
\end{aligned}
\end{equation*}
for all $t\leq \tau$ almost surely. Let $M\in\bN$ be specified later. Then, chain rule, integration with respect to $x$, and integration by parts imply
\begin{equation} \label{applying_chain_rule_noise_cancelling}
\begin{aligned}
& e^{-M\tau\wedge t}\int_{\bR}\left|w_m^{(\ep)}(\tau\wedge t,x)\right|^p dx \\
& = -p\int_0^{\tau\wedge t} \int_\bR \left(\left|w_m^{(\ep)}(s,x)\right|^{p-1}\right)_{x}  \left(aw_{m}\right)_x^{(\ep)}(s,x)   e^{-Ms} dxds \\
& - p\int_0^{\tau\wedge t} \int_\bR \left(\left|w_m^{(\ep)}(s,x)\right|^{p-1}\right)_{x}  \left( (2a_x - b)w_m \right)^{(\ep)}(s,x)  e^{-Ms} dxds \\
& + p\int_0^{\tau\wedge t} \int_\bR \left|w_m^{(\ep)}(s,x)\right|^{p-1} \left( (a_{xx} - b_x + c) w_m \right)^{(\ep)}(s,x)  e^{-Ms} dxds \\
& - \frac{1}{2}p\int_0^{\tau\wedge t} \int_\bR \left(\left|w_m^{(\ep)}(s,x)\right|^{p-1}\right)_{x} \left( (w_m(s,\cdot) + v(s,\cdot))^{1+\lambda} \right)^{(\ep)}(x)  \bar{b}(s)e^{-Ms} dxds \\
& - M\int_0^{\tau\wedge t} \int_\bR \left|w_m^{(\ep)}(s,x)\right|^p e^{-Ms}dxds \\
\end{aligned}
\end{equation}
for all $t>0$ almost surely. Notice that \eqref{boundedness_of_deterministic_coefficients_space-time_white_noise} and \eqref{ellipticity_of_leading_coefficients_space-time_white_noise} imply
\begin{equation} \label{approximation_of_leading_coeffi_1_STWN_1}
\begin{aligned}
& \left| (aw_{m})_x^{(\ep)}(s,x) - a(s,x)w_{mx}^{(\ep)}(s,x) \right| \\
&\quad = \ep^{-1}\left|\int_{\bR}  ( a(s,x-\ep y) - a(s,x) )w_m(s,x-\ep y)\zeta'(y)  dy \right|\\
&\quad \leq N(K)\int_{\bR} |w_m(s,x-\ep y)||y\zeta'(y)|dy
\end{aligned}
\end{equation}
and
\begin{equation} \label{approximation_of_leading_coeffi_2_STWN_1}
\begin{aligned}
&-\int_\bR a(s,x)\left|w_m^{(\ep)}(s,x)\right|^{p-2} \left|w_{mx}^{(\ep)}(s,x)\right|^2 dx \\
&\quad \leq -K^{-1} \int_\bR \left|w_m^{(\ep)}(s,x)\right|^{p-2} \left|w_{mx}^{(\ep)}(s,x)\right|^2  dx.
\end{aligned}
\end{equation}
Thus, by \eqref{approximation_of_leading_coeffi_1_STWN_1}, \eqref{approximation_of_leading_coeffi_2_STWN_1}, and Young's inequality,
\begin{equation*}
\begin{aligned}
&-\int_\bR \left|w_m^{(\ep)}(s,x)\right|^{p-2}  w_{mx}^{(\ep)}(s,x) \left(aw_{m}\right)_x^{(\ep)}(s,x) dx \\
& \quad = -\int_\bR \left|w_m^{(\ep)}(s,x)\right|^{p-2}  w_{mx}^{(\ep)}(s,x) \left[ \left(aw_{m}\right)_x^{(\ep)}(s,x) - a(s,x)w_{mx}^{(\ep)}(s,x) \right] dx \\
&\quad\quad\quad- \int_\bR a(s,x) \left|w_m^{(\ep)}(s,x)\right|^{p-2}   \left| w_{mx}^{(\ep)}(s,x) \right|^2 dx \\
\end{aligned}
\end{equation*}
\begin{equation} \label{bound_of_leading_term_STWN_1}
\begin{aligned}
&\quad \leq N(K)\int_\bR \left|w_m^{(\ep)}(s,x)\right|^{p-2}  \left| w_{mx}^{(\ep)}(s,x) \right| \int_{\bR} |w_m(s,x-\ep y)||y\zeta'(y)|dy dx \\
&\quad\quad\quad- K^{-1}\int_\bR \left|w_m^{(\ep)}(s,x)\right|^{p-2}  \left| w_{mx}^{(\ep)}(s,x) \right|^2 dx\\
&\quad \leq N\int_\bR \left|w_m^{(\ep)}(s,x)\right|^{p-2}  \left( \int_{\bR} |w_m(s,x-\ep y)||y\zeta'(y)|dy\right)^2 dx \\
&\quad\quad\quad- \frac{1}{2}K^{-1}\int_\bR \left|w_m^{(\ep)}(s,x)\right|^{p-2}  \left| w_{mx}^{(\ep)}(s,x) \right|^2 dx.
\end{aligned}
\end{equation}
Besides, we have
\begin{equation*}
\begin{aligned}
\left| \left(\left(2a_x - b\right)w_m\right)^{(\ep)}(s,x) \right| &= \left| \int_{\bR}\left(2a_y(s,y) - b(s,y)\right)w_m(s,y)\zeta_\ep(x-y)dy \right| \\
&\leq K\int_{\bR}|w_m(s,y)|\zeta_\ep(x-y)dy \\
&= K(|w_m(s,\cdot)|)^{(\ep)}(x)
\end{aligned}
\end{equation*}
and $\left| ((a_{xx} - b_x + c)w_m)^{(\ep)}(s,x) \right| \leq K(|w_m(s,\cdot)|)^{(\ep)}(x)$. Thus, by Young's inequality, we have
\begin{equation} \label{bound_of_lower_order_term_coefficient_STWN_1}
\begin{aligned}
&- p(p-1)\int_0^{\tau\wedge t} \int_\bR \left|w_m^{(\ep)}(s,x)\right|^{p-2}w_{mx}^{(\ep)}(s,x)  \left( (2a_x - b)w_m \right)^{(\ep)}(s,x)  e^{-Ms} dxds \\
&\quad + p\int_0^{\tau\wedge t} \int_\bR \left|w_m^{(\ep)}(s,x)\right|^{p-1} \left((a_{xx} - b_x + c) w_m \right)^{(\ep)}(s,x)  e^{-Ms} dxds \\
&\quad\quad\leq  \frac{1}{8}K^{-1}p(p-1)\int_0^{\tau\wedge t} \int_\bR \left|w_m^{(\ep)}(s,x)\right|^{p-2}\left|w_{mx}^{(\ep)}(s,x)\right|^2  e^{-Ms} dxds \\
&\quad\quad\quad + N\int_0^{\tau\wedge t} \int_\bR \left|w_m^{(\ep)}(s,x)\right|^{p-2}\left| \left( |w_m(s,\cdot)| \right)^{(\ep)}(x) \right|^2  e^{-Ms} dxds\\
&\quad\quad\quad + N\int_0^{\tau\wedge t} \int_\bR \left|w_m^{(\ep)}(s,x)\right|^{p-1} \left( |w_m(s,\cdot)| \right)^{(\ep)}(x)   e^{-Ms} dxds.
\end{aligned}
\end{equation}
Furthermore, since $p-1+\lambda\in2\bN$, we have 
\begin{equation*}
\begin{aligned}
\int_\bR \left| w_m^{(\ep)}(s,x) \right|^{p-1+\lambda}w_{mx}^{(\ep)}(s,x) dx &= \int_\bR \left( w_m^{(\ep)}(s,x) \right)^{p-1+\lambda}w_{mx}^{(\ep)}(s,x) dx \\
&= \frac{1}{p+\lambda}\int_\bR \left(\left( w_m^{(\ep)}(s,x) \right)^{p+\lambda}\right)_x dx \\
&= 0
\end{aligned}
\end{equation*}
for $s\leq \tau$. Thus, Young's inequality yields that
\begin{equation} \label{bound_of_nonlienar_term_STWN_1}
\begin{aligned}
&\int_\bR \left| w_m^{(\ep)}(s,x) \right|^{p-2}w_{mx}^{(\ep)}(s,x)\left( (w_m(s,\cdot) + v(s,\cdot))^{1+\lambda} \right)^{(\ep)}(x) dx \\
& = -\int_\bR \left| w_m^{(\ep)}(s,x) \right|^{p-2}w_{mx}^{(\ep)}(s,x) \Phi_\ep(w_m,v,s,x)dx \\
& \leq \frac{1}{4K}\int_\bR \left| w_m^{(\ep)}(s,x) \right|^{p-2}  \left| w_{mx}^{(\ep)}(s,x) \right|^2  dx + N\int_\bR \left| w_m^{(\ep)}(s,x) \right|^{p-2}   \Phi_\ep^2(w_m,v,s,x)dx,
\end{aligned}
\end{equation}
where $\Phi_\ep(w_m,v,s,x) = \left( (w_m(s,\cdot) + v(s,\cdot))^{1+\lambda}  \right)^{(\ep)}(x) - \left| w_m^{(\ep)}(s,x) \right|^{1+\lambda}$.
Therefore, applying \eqref{bound_of_leading_term_STWN_1}, \eqref{bound_of_lower_order_term_coefficient_STWN_1}, and \eqref{bound_of_nonlienar_term_STWN_1} to \eqref{applying_chain_rule_noise_cancelling}, we have
\begin{equation*}
\begin{aligned}
&  e^{-M\tau\wedge t}\int_{\bR}\left|w_m^{(\ep)}(\tau\wedge t,x)\right|^p dx \\
&\quad\leq N\int_0^{\tau\wedge t}\int_\bR \left|w_m^{(\ep)}(s,x)\right|^{p-2}  \left( \int_{\bR} |w_m(s,x-\ep y)||y\zeta'(y)|dy\right)^2 dx ds \\
&\quad\quad + N\int_0^{\tau\wedge t} \int_\bR \left|w_m^{(\ep)}(s,x)\right|^{p-2}  \left| \left( |w_m(s,\cdot)| \right)^{(\ep)}(x) \right|^2  e^{-Ms} dxds \\
&\quad\quad + N\int_0^{\tau\wedge t} \int_\bR \left|w_m^{(\ep)}(s,x)\right|^{p-1} \left(|w_m(s,\cdot)|\right)^{(\ep)}(x)  e^{-Ms} dxds \\
&\quad\quad +  N\int_0^{\tau\wedge t}\int_\bR \left| w_m^{(\ep)}(s,x) \right|^{p-2}   \Phi^2_\ep(w_m,v,s,x) e^{-Ms} dxds \\
&\quad\quad - M\int_0^{\tau\wedge t} \int_\bR \left|w_m^{(\ep)}(s,x)\right|^p e^{-Ms} dxds.
\end{aligned}
\end{equation*}
Thus, by letting $\ep\downarrow0$, we have
\begin{equation*}
\begin{aligned}
& e^{-M\tau\wedge t}\int_{\bR}\left|w_m(\tau\wedge t,x)\right|^p dx \\
&\leq  (N-M)\int_0^{\tau\wedge t} \int_\bR \left|w_m(s,x)\right|^{p} e^{-Ms} dxds \\
&+  N\int_0^{\tau\wedge t}\int_\bR \left| w_m(s,x) \right|^{p-2}   \left( (w_m(s,x) + v(s,x))^{1+\lambda} - (w_m(s,x))^{1+\lambda} \right)^2 e^{-Ms} dxds  \\
\end{aligned}
\end{equation*}
Choose $M = 2N$. Since $|a^{1+\lambda} - b^{1+\lambda}| \leq N(\lambda)|a-b|(a^{\lambda}+b^{\lambda})$ for $\lambda,a,b\in[0,\infty)$,  Young's inequality implies that
\begin{equation*}
\begin{aligned}
& e^{-M\tau\wedge t}\int_{\bR}\left|w_m(\tau\wedge t,x)\right|^p dx \\
&\leq N\int_0^{\tau\wedge t}\int_\bR \left| w_m(s,x) \right|^{p-2}   \left| (w_m(s,x) + v(s,x))^{1+\lambda} - (w_m(s,x))^{1+\lambda} \right|^2 e^{-Ms} dxds \\
&\leq N\int_0^{\tau\wedge t}\int_\bR \left| w_m(s,x) \right|^{p-2}    (2|w_m(s,x)|^\lambda + |v(s,x)|^\lambda) |v(s,x)| e^{-Ms} dxds \\
&\leq  N\int_0^{\tau\wedge t}  \int_\bR | w_m(s,x) |^{p} e^{-M\tau \wedge s} dxds \\
&\quad+ N\int_0^{\tau}\int_\bR ( | v(s,x) |^{\frac{p}{2-\lambda}} + | v(s,x) |^{\frac{p(1+\lambda)}{2}}) e^{-Ms} dxds \\
&\leq N\int_0^{t}\int_\bR | w_m(\tau\wedge s,x) |^{p}e^{-M \tau\wedge s}dx  ds \\
&\quad + N\left(\sup_{t\leq\tau,x\in\bR}|v(t,x)|^{\frac{p-2+\lambda}{2-\lambda}} + \sup_{t\leq\tau,x\in\bR}|v(t,x)|^{\frac{p(1+\lambda)-2}{2}}\right)\int_0^{\tau}\int_\bR | v(s,x) |  e^{-Ms} dxds
\end{aligned}
\end{equation*}
for all $t>0$ almost surely. 
Recall that $w_m := u_m - v\in\cH_p^{1/2-\kappa}(\tau_m^m)$. By Theorem \ref{embedding} \eqref{large-p-embedding}, we have $w_m\in C([0,\tau_m^m];L_p)$ (a.s.).
For almost sure $\omega$, by Gr\"onwall's inequality, we have
\eqref{Lq_bound_2_estimate}. The lemma is proved.

\qed
\end{proof}

By combining Lemmas \ref{noise_cancelling_lemma} and \ref{Lq_bound_2}, we show that solution candidate $u$ is non-explosive.

\begin{lemma} \label{Non_explosion_large_lambda}
All the conditions of Theorem \ref{theorem_burgers_eq_large_lambda_STWN} hold. Let $u$ be the function introduced in Remark \ref{def_u}. Then, for any $T<\infty$, we have
\begin{equation} 
\label{stopping_time_blow_up_STWN_1}
\lim_{R\to\infty} P\left( \left\{ \omega\in\Omega:\sup_{t\leq T,x\in\bR} |u(t,x)| > R \right\} \right) = 0.
\end{equation}
\end{lemma}
\begin{proof}
Let $T<\infty$. Suppose $v$ is the solution to \eqref{noise_cancelling_eq}  introduced in Lemma \ref{noise_cancelling_lemma}.
For $m,S>0$, define
\begin{equation*}
\begin{aligned}
\tau^1(S) &:=\inf\Bigg\{ t\geq 0:\| v(t,\cdot) \|_{L_{1}}\geq S \Bigg\}, \quad \tau^2(S) :=\inf\Bigg\{ t\geq 0:\sup_{x\in\bR}|v(t,x)|\geq S \Bigg\},
\end{aligned}
\end{equation*}
and
\begin{equation*}
\tau := \tau(m,S,T) := \tau_{m}^m\wedge\tau^1(S) \wedge \tau^2(S)\wedge T,
\end{equation*}
where $\tau_m^m$ is defined as in \eqref{stopping_time_taumm}.
By Lemma \ref{noise_cancelling_lemma}, $\tau^1(S)$ and $\tau^2(S)$ are well-defined stopping times and thus $\tau$ is a stopping time.  Then, by Theorem \ref{theorem_nonlinear_case}, H\"older's inequality, and Minkowski's inequality, we have
\begin{equation}
\label{main_computation_Non_explosion_small_lambda}
\begin{aligned}
&\| u_m \|^p_{\cH_{p}^{1/2-\kappa}(\tau )} - N\| u_0 \|^p_{U_p^{1/2-\kappa}}\\
&\leq N\left\| \bar{b}\left(u_{m+}^{1+\lambda} h_m(u_m)\right)_{x} \right\|_{\bH_p^{-3/2-\kappa}(\tau)} + N\| \sigma(u_m)\eta \|_{\bH_p^{-1/2-\kappa}(\tau,\ell_2)} \\
&\leq N\left\|  u_{m+}^{1+\lambda} h_m(u_m) \right\|_{\bH_p^{-1/2-\kappa}(\tau)} + N\| \sigma(u_m)\eta \|_{\bH_p^{-1/2-\kappa}(\tau,\ell_2)} \\
&\leq N\bE\int_0^{\tau }\int_{\bR} \left| \int_{\bR} R_{1/2+\kappa}(y)|u_m(s,x-y)|^{1+\lambda} dy \right|^p dxds\\
&\quad\quad + N\bE\int_0^{\tau }\int_{\bR}\left| \int_{\bR} \left| R_{1/2+\kappa}(y) \right|^2|u_m(s,x-y)|^{2} dy \right|^{p/2} dxds \\
&\leq N\bE\int_0^{\tau }  \int_{\bR} \left| \int_{\bR} \left| R_{1/2+\kappa}(y) \right|^{\frac{p}{p-\lambda}}|u_m(s,x-y)|^{\frac{p}{p-\lambda}} dy \right|^{p-\lambda} dx \left(\int_{\bR} |u_m(s,x)|^{p}  dx\right)^\lambda ds \\
&\quad\quad + N\bE\int_0^{\tau }\left( \int_{\bR} \left|R_{1/2+\kappa}(x)\right|^2 dx\right)^{p/2}  \int_{\bR} |u_m(s,x)|^{p}  dx  ds \\
&\leq N\bE\int_0^{\tau } \left( \int_{\bR} \left| R_{1/2+\kappa}(y) \right|^{\frac{p}{p-\lambda}} dy \right)^{p-\lambda} \left(\int_{\bR} |u_m(s,x)|^{p}  dx\right)^{1+\lambda}ds \\
&\quad\quad + N\bE\int_0^{\tau }\left( \int_{\bR} \left|R_{1/2+\kappa}(x)\right|^2 dx\right)^{p/2}  \int_{\bR} |u_m(s,x)|^{p}  dx  ds \\
\end{aligned}
\end{equation}
where $N = N(p,\kappa,K,T)$. Since $\frac{p}{p-\lambda}\left(\frac{1}{2}-\kappa\right)<1$ and $2\left(\frac{1}{2}-\kappa\right)<1$, by Remark \ref{Kernel}, we have
\begin{equation}
\label{main_computation_Non_explosion_small_lambda_2}
\left( \int_{\bR} \left| R_{1/2+\kappa}(y) \right|^{\frac{p}{p-\lambda}} dy \right)^{p-\lambda} + \left(\int_{\bR} \left|R_{1/2+\kappa}(x)\right|^2 dx\right)^{p/2} < \infty.
\end{equation}
Besides, since $s\leq \tau$ and $\tau \leq \tau_m^m$, Lemma \ref{Lq_bound_2} yields
\begin{equation}
\label{main_computation_Non_explosion_small_lambda_3}
\begin{aligned}
\int_\bR |u_m(s,x)|^p dx 
&\leq N(p)\left(\int_\bR |u_m(s,x) - v(s,x)|^p dx + \int_\bR |v(s,x)|^p dx\right) \\
&\leq N(p,K,T)\Theta(v)\sup_{s\leq \tau} \| v(s,\cdot) \|_{L_1}  \\
&\leq N(S,p,\lambda,K,T),
\end{aligned}
\end{equation}
for all $s\leq \tau$ almost surely, where
$$\Theta(v) = \sup_{s\leq \tau,x\in\bR}|v(s,x)|^{\frac{p+\lambda-2}{2-\lambda}} + \sup_{s\leq \tau,x\in\bR}|v(s,x)|^{\frac{p(1+\lambda)-2}{2}} + \sup_{s\leq\tau,x\in\bR}|v(s,x)|^{p-1}.
$$
Therefore, by \eqref{main_computation_Non_explosion_small_lambda}, \eqref{main_computation_Non_explosion_small_lambda_2}, and \eqref{main_computation_Non_explosion_small_lambda_3}, we have
\begin{equation}
\label{bound_of_local_sol}
\| u_m \|^p_{\cH_{p}^{1/2-\kappa}(\tau)} \leq N(p,\kappa,K,T)\| u_0 \|^p_{U_p^{1/2-\kappa}} + N(S,p,\kappa,\lambda,K,T).
\end{equation}
It should be remarked that $N$ is independent of $m$. Therefore, by Corollary \ref{embedding_corollary}, $\bE\sup_{t\leq \tau,x}|u_m(t,x)|^p$ is bounded by the right hand side of \eqref{bound_of_local_sol}. Thus, for any $R>0$, by Chebyshev's inequality and \eqref{bound_of_local_sol}, we have
\begin{equation}
\label{chebyshev's_ineq_1}
\begin{aligned}
P\left( \sup_{t\leq \tau,x\in\bR}|u(t,x)| > R \right) &\leq \frac{1}{R^p}\bE\sup_{t\leq \tau,x\in\bR}|u(t,x)|^p \\
&= \frac{1}{R^p}\bE\sup_{t\leq \tau,x\in\bR}|u_m(t,x)|^p \\
&\leq \frac{1}{R^p}\| u_m \|_{\cH_p^{1/2-\kappa}(\tau)}^p \\
&\leq \frac{1}{R^p}N(u_0,S,p,\kappa,\lambda,K,T).
\end{aligned}
\end{equation}
On the other hand, by Chebyshev's inequality and Lemma \ref{noise_cancelling_lemma}, we have
\begin{equation}
\label{chebyshev's_ineq_2}
\begin{aligned}
&P\left(\tau^1(S)<T\right) + P\left(\tau^2(S)<T\right)\\
&\quad\leq P\left( \sup_{t\leq T}\| v(t,\cdot)\|_{L_{1}} > S \right) + P\left( \sup_{t\leq T,x\in\bR} |v(t,x)| > S \right) \\
&\quad\leq \frac{1}{\sqrt{S}}\bE\sup_{t\leq T}\| v \|_{L_{1}}^{1/2} + \frac{1}{S^p}\bE\sup_{t\leq T,x\in\bR} |v(t,x)|^p \\
&\quad\leq \frac{1}{\sqrt S}N(u_0,p,\kappa,K,T).
\end{aligned}
\end{equation}
Thus, Chebyshev's inequality, Fatou's lemma, \eqref{chebyshev's_ineq_1}, and \eqref{chebyshev's_ineq_2} yield
\begin{equation*}
\begin{aligned}
P\left(\sup_{t\leq T,x\in\bR}|u(t,x)| > R\right) 
= P\left(\sup_{t\leq  \tau_\infty\wedge T,x\in\bR}|u(t,x)| > R\right)
\end{aligned}
\end{equation*}
\begin{equation*}
\begin{aligned}
&\leq P\left(\sup_{t\leq \tau_\infty\wedge \tau^1(S)\wedge \tau^2(S_2)\wedge T,x\in\bR}|u(t,x)| > R\right) + P\left(\tau^1(S)<T\right) + P\left(\tau^2(S)<T\right)\\
&\leq \liminf_{m\to\infty}P\left(\sup_{t\leq \tau(m,S,T),x\in\bR}|u(t,x)| > R\right) + P\left(\tau^1(S)<T\right) + P\left(\tau^2(S)<T\right)  \\
&\leq \frac{N_1}{R^p} + \frac{N_2}{\sqrt S},
\end{aligned}
\end{equation*}
where $N_1 = N_1(u_0,S,p,\kappa,\lambda,K,T)$ and $N_2 = N_2(u_0,p,\kappa,K,T)$. Notice that $N_1$ is independent of  $R$, and $N_2$ is independent of $R$ and $S$. By letting $R\to\infty$, and $S\to\infty$ in order, we get \eqref{stopping_time_blow_up_STWN_1}. The lemma is proved.

\qed
\end{proof}

\noindent\textbf{Proof of Theorem \ref{theorem_burgers_eq_large_lambda_STWN}} 
{\it Step 1. (Uniqueness). }
Suppose $u,\bar u\in \cH_{p,loc}^{1/2-\kappa}$ are nonnegative solutions of equation \eqref{burger's_eq_space_time_white_noise}. By Definition \ref{definition_of_sol_space} \eqref{def_of_local_sol_space}, there are bounded stopping times $\tau_n$ $(n = 1,2,\cdots)$ 
such that
\begin{equation*}
\tau_n\uparrow\infty\quad\mbox{and}\quad u, \bar u \in \cH_{p}^{1/2-\kappa}(\tau_n).
\end{equation*}
Fix $n\in\bN$. Since $p > \frac{6}{1-2\kappa}$, by Corollary \ref{embedding_corollary}, we have $u,\bar{u} \in C([0,\tau_n];C(\bR))$ (a.s.) and
\begin{equation}
\label{embedding in the proof of uniqueness_infinite_noise_1}
\bE\sup_{t\leq\tau_n}\sup_{x\in\bR}|u(t,x)|^p + \bE\sup_{t\leq\tau_n}\sup_{x\in\bR}|\bar u(t,x)|^p < \infty.
\end{equation}
For $m \in \bN$, define
\begin{equation*}
\begin{gathered}
\tau_{m,n}^1:=\inf\left\{t\geq0:\sup_{x\in\bR}|u(t,x)|> m\right\}\wedge\tau_n, \\
\tau_{m,n}^2:=\inf\left\{t\geq0:\sup_{x\in\bR}|\bar u(t,x)|> m\right\}\wedge\tau_n,
\end{gathered}
\end{equation*}
and 
\begin{equation} 
\label{stopping_time_cutting}
\tau_{m,n}:=\tau_{m,n}^1\wedge\tau_{m,n}^2.
\end{equation}
Due to \eqref{embedding in the proof of uniqueness_infinite_noise_1}, $\tau_{m,n}^1$ and $\tau_{m,n}^2$ are well-defined stopping times, and thus $\tau_{m,n}$ is a stopping time.
Observe that $u,\bar u\in\cH_p^{1/2-\kappa}(\tau_{m,n})$ and  $\tau_{m,n}\uparrow \tau_n$ as $m\to\infty$ almost surely. Fix $m\in\bN$. Notice that $u,\bar u\in\cH_p^{1/2-\kappa}(\tau_{m,n})$ are solutions to equation
\begin{equation*}
dv = \left(  av_{xx} + b v_{x} + cv + \bar{b} \left( v_+^{1+\lambda}h_m(v) \right)_{x}  \right)\,dt + \sigma^k(v) \eta_k dw^k_t, \quad 0<t\leq\tau_{m,n}
\end{equation*}
with the initial data $v(0,\cdot)=u_0$.
By the uniqueness result in Lemma \ref{cut_off_lemma_large_lambda_STWN}, we conclude that $u=\bar u$ in $\cH_p^{1/2-\kappa}(\tau_{m,n})$ for each $m\in\bN$. The monotone convergence theorem yields $u=\bar u$ in $\cH_p^{1/2-\kappa}(\tau_n)$ and this implies $u = \bar u$ in $\cH_{p,loc}^{1/2-\kappa}$.
\\

{\it Step 2. (Existence).}
For $m\in\bN$, define a stopping time $\tau_m^m$ and $u$ as in Remark \ref{def_u}. Let $T<\infty$. Observe that
\begin{equation*}
\left\{\omega\in\Omega : \tau_m^m < T \right\} \subset \left\{\omega\in\Omega : \sup_{t \leq \tau_m^m,x\in\bR}|u(t,x)| \leq \sup_{t \leq T,x\in\bR}|u(t,x)| \right\}.
\end{equation*}
Since $\sup_{t \leq \tau_m^m,x\in\bR}|u(t,x)| = \sup_{t \leq \tau_m^m,x\in\bR}|u_m(t,x)| = m$ (a.s.), by Lemma \ref{Non_explosion_large_lambda}, 
\begin{equation*}
\begin{aligned}
\limsup_{m\to\infty}  P\left( \tau_m^m < T \right) 
&\leq \limsup_{m\to\infty} P\left(\sup_{t\leq T,x\in\bR}|u(t,x)|\geq m\right)\to 0.
\end{aligned}
\end{equation*}
Since $T<\infty$ is arbitrary, this implies $\tau_m^m\to \infty$ in probability. Since $\tau_m^m$ is increasing in $m$,  we conclude that $\tau_m^m \uparrow \infty$ (a.s.). 

Lastly, Set $\tau_m:=\tau_m^m \wedge m$. Recall that 
\begin{equation*}
u(t,x):=u_m(t,x)\quad \text{for}\quad t\in[0,\tau_m^m].
\end{equation*}
Observe that $|u_m(t)|\leq m$ for $t\in[0,\tau_m^m]\cap[0,\infty)$ and thus $u_m$ satisfies \eqref{burger's_eq_space_time_white_noise} for all $t\in[0,\tau_m^m]\cap[0,\infty)$ (a.s.).
Since $u=u_m$ for $t\in[0,\tau_m^m]\cap[0,\infty)$ and $u_m\in \cH_p^{1/2-\kappa}(T)$ for any $T<\infty$, it follows that $u\in\cH_p^{1/2-\kappa}(\tau_m)$ and $u$ satisfies \eqref{burger's_eq_space_time_white_noise} for all $t \leq \tau_m$ (a.s.). Since $\tau_m\uparrow \infty$ (a.s.) as $m\to\infty$, we have $u\in\cH_{p,loc}^{1/2-\kappa}$.

{\it Step 3. (H\"older regularity). } Let $T<\infty$. Since $u\in \cH_p^{1/2-\kappa}(T)$, by employing Corollary \ref{embedding_corollary}, we have \eqref{holder_regularity_large_lambda_STWN}. The theorem is proved.

\qed

\noindent\textbf{Proof of Theorem \ref{uniqueness_in_p_1}}
The proof of Theorem \ref{uniqueness_in_p_1} is motivated by \cite[Corollarly 5.11]{kry1999analytic}. Let $q>p$. By Theorem \ref{theorem_burgers_eq_large_lambda_STWN}, there exists a unique solution $\bar{u}\in \cH_{q,loc}^{1/2-\kappa}$ satisfying equation \eqref{burger's_eq_space_time_white_noise}. By Definition \ref{definition_of_sol_space} \eqref{def_of_local_sol_space}, there exists $\tau_n$ such that $\tau_n\to\infty$ (a.s.) as $n\to\infty$, $u\in\cH_p^{1/2-\kappa}(\tau_n)$ and $\bar{u} \in\cH_q^{1/2-\kappa}(\tau_n)$. Fix $n\in\bN$. Since $\frac{6}{1-2\kappa} < p < q$, we can define $\tau_{m,n}$ $(m\in\bN)$ as in \eqref{stopping_time_cutting}. For any $p_0 > p$, we have
\begin{equation*}
u\in \bL_{p_0}(\tau_{m,n}) 
\end{equation*}
since
\begin{equation*}
\bE\int_0^{\tau_{m,n}} \int_{\bR} |u(t,x)|^{p_0} dxdt \leq m^{p_0-p}\bE\int_0^{\tau_{m,n}} \int_{\bR} |u(t,x)|^{p} dxdt < \infty.
\end{equation*} 
Observe that $\bar{b} \left( u^{1+\lambda} \right)_{x} \in \bH_q^{-1}(\tau_{m,n})$. Indeed, by \eqref{boundedness_of_deterministic_coefficients_space-time_white_noise}, Lemma \ref{prop_of_bessel_space} \eqref{bounded_operator}, \eqref{pointwise_multiplier}, we have
\begin{equation*}
\begin{aligned}
&\bE \int_0^{\tau_{m,n}} \left\| \bar{b}(s)  \left( (u(s,\cdot))^{1+\lambda} \right)_{x} \right\|_{H_q^{-1}}^q ds  \leq N \bE \int_0^{\tau_{m,n}} \int_{\bR}|u(s,x)|^{q(1+\lambda)} dxds < \infty.
\end{aligned}
\end{equation*}
Also, by Lemma \ref{prop_of_bessel_space} \eqref{pointwise_multiplier} again, we have
\begin{equation*}
\begin{aligned}
&au_{xx} \in \bH_{q}^{-2}(\tau_{m,n}), \quad b u_{x} \in \bH_{q}^{-1}(\tau_{m,n}), \quad\text{and}\quad c u \in \bL_{q}(\tau_{m,n}).
\end{aligned}
\end{equation*}
Therefore, since $\bL_{q}(\tau_{m,n})\subset \bH_{q}^{-1}(\tau_{m,n}) \subset \bH_{q}^{-2}(\tau_{m,n})$, Lemma \ref{prop_of_bessel_space} \eqref{norm_bounded} yields
\begin{equation} \label{deterministic_part_of_u_infinite_noise_1}
au_{xx} + b u_{x} + c u + b\left(u^{1+\lambda} \right)_{x} \in \bH_{q}^{-2}(\tau_{m,n}).
\end{equation}
By \eqref{boundedness_of_stochastic_coefficients_large_lambda_STWN},
\begin{equation} \label{proof_of_claim_stochastic_part}
\begin{aligned}
\|\sigma(u)\|^q_{\bL_q(\tau_{m,n},\ell_2)} &= \bE\int_0^{\tau_{m,n}} \int_{\bR} \left(\sum_{k} \left| \sigma^k(s,x,u(s,x)) \right|^2\right)^{q/2} dxds \\
&\leq N\left\| u \right\|_{\bL_q(\tau_{m,n})}^q \\
&< \infty.
\end{aligned}
\end{equation}
Thus, we have
\begin{equation} \label{stochastic_part_of_u_infinite_noise_1}
\sigma( u ) \in \bL_q(\tau_{m,n},\ell_2) \subset \bH_{q}^{-1}(\tau_{m,n},\ell_2).
\end{equation}
Due to \eqref{deterministic_part_of_u_infinite_noise_1} and \eqref{stochastic_part_of_u_infinite_noise_1}, the right hand side of equation \eqref{burger's_eq_space_time_white_noise} is $H_{p}^{-2}$-valed continuous function on $[0,\tau_{m,n}]$. Thus, $u$ is in $\cL_q(\tau_{m,n})$ and $u$ satisfies \eqref{burger's_eq_space_time_white_noise} for all $t\leq \tau_{m,n}$ almost surely with $u(0,\cdot) = u_0(\cdot)$. On the other hand, since $\bar{b} \left( u^{1+\lambda} \right)_{x} \in \bH_q^{-3/2-\kappa}(\tau_{m,n})$ and $\sigma( u ) \in \bH^{-1/2-\kappa}_q(\tau_{m,n},\ell_2) $, Theorem \ref{theorem_nonlinear_case} implies that there exists $v\in \cH_q^{1/2-\kappa}(\tau_{m,n})$ satisfying
\begin{equation}
\label{equation_in_proof_of_consistency_infinite_noise_1}
dv = \left(av_{xx} + b v_{x} + cv +  \bar{b}\left( u^{1+\lambda}  \right)_{x}  \right) dt + \sigma^k(u)\eta_k dw_t^k, \quad 0<t\leq\tau_{m,n} 
\end{equation}
with the initial data $v(0,\cdot)=u_0.$
In \eqref{equation_in_proof_of_consistency_infinite_noise_1}, $\bar{b}\left( u^{1+\lambda}  \right)_{x}$ and $\sigma^k(u)$ are used instead of $\bar{b}\left( v^{1+\lambda}  \right)_{x}$ and $\sigma^k(v)$. Since $u\in\cL_q(\tau_{m,n})$ satisfies equation \eqref{equation_in_proof_of_consistency_infinite_noise_1}, $w := u-v\in \cL_q(\tau_{m,n})$ satisfies 
\begin{equation*}
dw = \left(aw_{xx} + b w_{x} + cw \right) dt, \quad 0<t\leq\tau_{m,n}\,; \quad w(0,\cdot)=0. 
\end{equation*}
By Theorem \ref{theorem_nonlinear_case}, we have $w = 0$ and thus $u = v$ in $\cL_q(\tau_{m,n})$. Therefore, $u$ is in $\cH_q^{1/2-\kappa}(\tau_{m,n})$. Note that $\bar{u}\in \cH_q^{1/2-\kappa}(\tau_{m,n})$ satisfies equation \eqref{burger's_eq_space_time_white_noise}. By Theorem \ref{cut_off_lemma_large_lambda_STWN}, we have $u = \bar{u}$ in $\cH_q^{1/2-\kappa}(\tau_{m,n})$.  The theorem is proved.
\qed


\section{Proof of the second case: modified Burgers' equation with the super-linear diffusion coefficient \texorpdfstring{$\sigma(u)$}{Lg} } 
\label{Proof of the second case}

This section proposes proof of Theorems \ref{theorem_burgers_eq_STWN_super_linear} and \ref{uniqueness_in_p_2}. The main idea of the proof for Theorem \ref{theorem_burgers_eq_STWN_super_linear} is similar to the proof of Theorem \ref{theorem_burgers_eq_large_lambda_STWN}. However, we deal with the nonlinear terms $ |u|^\lambda u_x$ and $|u|^{1+\lambda_0}$ simultaneously. Intuitively, we understand equation \eqref{burger's_eq_space_time_white_noise_super_linear} as
\begin{equation}
\label{example_eq_5}
du = \left(a^{ij}u_{x^ix^j} + b^i u_{x^i} + cu + \bar{b}^i\left( \xi u \right)_{x^i}\right) dt + \mu\, \xi_0 \,u\,\eta_k dw_t^k,\quad t > 0\,; \quad u(0) = u_0,
\end{equation}
where $\xi = |u|^{1+\lambda}/u$ and $\xi_0 = |u|^{1+\lambda_0}/u$ $(0/0:=0)$. 
Then, Theorem \ref{theorem_nonlinear_case} and \eqref{main_computation_Non_explosion_small_lambda} yield
\begin{align*}
\|u\|_{\cH_p^{1/2-\kappa}(\tau)}&\leq N\left(\|u_0\|_{U_p^{1/2-\kappa}} + \| \bar b \xi u \|_{H_p^{-3/2-\kappa}} + \| \mu \xi_0 u \|_{H_p^{-1/2-\kappa}(\ell_2)}\right) \\
&\leq N\|u_0\|_{U_p^{1/2-\kappa}} + N \left(\| \xi \|_{L_s}^{1/s} + \| \xi_0 \|_{L_{2s_0}}^{1/{2s_0}}\right)\| u \|_{L_p}
\end{align*}
where $s>1$ and $s_0>1$. Note that the coefficients $\xi\in L_s$ and $\xi_0\in L_{s_0}$ should be controlled to extend the local solution to a global one. Since we have the uniform $L_1$ bound of $u_m$ (Lemma \ref{L_1_bound_small_lambda_STWN_lemma}), the case $s = 1/\lambda$ and $s_0 = 1/(2\lambda_0)$ is considered. Therefore, the conditions on $\lambda\in(0,1)$ and $\lambda_0\in(0,1/2)$ are required. With the assumptions on $\lambda$ and $\lambda_0$, the non-explosive property of local solutions is proved; see Lemma \ref{Non_explosion_small_lambda_STWN}.

\vspace{1mm}

Recall that $h(\cdot)\in C_c^\infty(\bR)$ satisfies $h\geq0$, $h(z) = 1$ on $|z|\leq 1$, and $h(z) = 0$ on $|z|\geq2$. Besides, 
$$ h_m(z) := h(z/m).
$$

The following lemma provides the unique existence of the local solution.

\begin{lemma} \label{cut_off_lemma_small_lambda_STWN}
Let $\lambda,\lambda_0\in(0,\infty)$, $T\in(0,\infty)$, $\kappa \in (0,1/2)$, and $p > \frac{6}{1-2\kappa}$. Suppose Assumptions \ref{deterministic_part_assumption_on_coeffi_space-time_white_noise} and \ref{stochastic_part_assumption_on_coeffi_space_time_white_noise_small_lambda} hold. Then, for a bounded stopping time $\tau\leq T$, $m\in\bN$, and nonnegative initial data $u_0\in U_{p}^{1/2-\kappa}$, there exists a unique $u_m \in \cH_p^{1/2-\kappa}(\tau)$ such that $u_m$ satisfies equation 
\begin{equation} 
\label{cut_off_equation_small_lambda_STWN}
du = \left(au_{xx} + b u_{x} + cu +  \bar{b}\left(u_+^{1+\lambda }h_m(u) \right)_{x}  \right) dt + \mu u_+^{1+\lambda_0}h_m(u)\eta_k dw_t^k,
\end{equation}
on $(t,x)\in(0,\tau)\times\bR$ with the initial data $u(0,\cdot) = u_0(\cdot)$. Furthermore, $u_m\geq0$. 
\end{lemma}
\begin{proof}
Follow the proof of Lemma \ref{cut_off_lemma_large_lambda_STWN}. However, instead of \eqref{l_2_computation}, we apply
\begin{equation*}
\begin{aligned}
&\left\| \mu u_+^{1+\lambda_0}h_m(u)\eta - \mu v_+^{1+\lambda_0}h_m(v)\eta \right\|_{H_p^{-1/2-\kappa}(\ell_2)}^p \\
& \leq \int_\bR  \left( \sum_k \left( \int_\bR \left| R_{1/2+\kappa}(x-y) \right|\mu(s,y) \bar{\Phi}_{m,\lambda_0}(u,v,s,y)\eta_k(y) dy \right)^2 \right)^{p/2} dx \\
& \leq N_m\int_\bR  \left( \int_\bR \left|R_{1/2+\kappa}(y)\right|^2(u(s,x-y) - v(s,x-y))^2 dy \right)^{p/2} dx \\
& \leq N_m\left( \int_\bR \left| R_{1/2+\kappa}(y) \right|^2 dy \right)^{p/2}\int_\bR  |u(s,x) - v(s,x)|^p  dx,
\end{aligned}
\end{equation*}
where $\bar{\Phi}_{m,\lambda_0}(u,v,s,y) = \left(u_+^{1+\lambda_0}(s,y)h_m(u(s,y)) - v_+^{1+\lambda_0}(s,y)h_m(v(s,y))\right)$.
The lemma is proved.

\qed
\end{proof}

\begin{remark}
Suppose \eqref{cut_off_equation_small_lambda_STWN} holds for all $t$, and $u_m$ is the solution Lemma \ref{cut_off_lemma_small_lambda_STWN}. Then, for any $\phi\in \cS$,
\begin{equation*}
M_t:=\sum_{k=1}^\infty \int_0^t \int_{\bR} \mu u_+^{1+\lambda_0} h_m(u) \eta_k \phi dx dw_s^k,\quad t<\infty
\end{equation*}
is a square integrable martingale. Indeed,
\begin{equation*}
\begin{aligned}
&\sum_{k=1}^\infty \int_0^t \left( \int_{\bR} \mu(s,x) u_+^{1+\lambda_0}(s,x) h_m(u(s,x)) \eta_k(x) \phi(x) dx \right)^2 ds \\
&\quad\leq N \sum_{k=1}^\infty \int_0^t \left( \int_{\bR}  \eta_k(x) \phi(x) dx \right)^2 ds \\
&\quad\leq N\| \phi \|_{L_2}^2 \\
&\quad< \infty
\end{aligned}
\end{equation*}
almost surely.
\end{remark}

\begin{lemma}
\label{L_1_bound_small_lambda_STWN_lemma}
Suppose all the conditions of Lemma \ref{cut_off_lemma_small_lambda_STWN} hold for $\tau = T$, and  we assume that $u_0\in  U_{p}^{1/2-\kappa}\cap L_1(\Omega;L_1)$. Let $u_m$ be the solution to equation \eqref{cut_off_equation_small_lambda_STWN} introduced in Lemma \ref{cut_off_lemma_small_lambda_STWN}. Then
\begin{equation} 
\label{L_1_bound_small_lambda_STWN}
\bE\sup_{t\leq T}\|u_m(t,\cdot)\|_{L_1}^{1/2} \leq N\|u_0\|_{L_1(\Omega\times\bR)}^{1/2}.
\end{equation} 
\end{lemma}
\begin{proof}
Follow the proof of Lemma \ref{L1_bound}. When we obtain \eqref{two kinds of L1bound}, however, employ \cite[Theorem III.6.8]{diffusion} instead of integration in $t$. Then, we have
\begin{equation*}
\bE\sup_{t\leq T}\left(\int_{\bR}u_m(t,x) h_k(x)dx\right)^{1/2} \leq 3e^{2 KT}\bE\|u_0\|_{L_1}^{1/2} + k^{-1/2+1/(2q)}N e^{2KT},
\end{equation*}
where $N = N(m,p,K,T)$. By letting $k\to\infty$, the monotone convergence theorem implies that 
\begin{equation} \label{L_1 est.}
\bE\sup_{t\leq T}\|u_m(t,\cdot)\|_{L_1}^{1/2} \leq 3e^{2KT}\bE\|u_0\|_{L_1}^{1/2}.
\end{equation}
The lemma is proved.
\qed
\end{proof}

\begin{remark}
To construct a global solution from the local ones, we need to show the local solution $u_m$ is non-explosive. Thus, we prove $\tau_m:=\tau_m^m\wedge m\to\infty$ (a.s.) as $m\to\infty$, where $\tau_m^R$ is a stopping time introduced in Remark \ref{def_u}. Then, the uniqueness of $u_m$ (Lemma \ref{cut_off_lemma_small_lambda_STWN}) implies that a global solution
$$ u(t,x):= u_{m}(t,x)\quad\text{for}\quad t\leq \tau_m
$$
is well-defined.

Note that we employ Lemma \ref{L_1_bound_small_lambda_STWN_lemma} to show 
\begin{equation}
\label{non_explosive_limit}
\tau_m^m\to\infty\quad \text{ as } \quad m\to\infty\quad(a.s.).
\end{equation}
Since $L_1$ norm of $u_m$ is uniformly bounded,  a stopping time $\tau_m(S):=\inf\{ t\geq0:\| u_m(t,\cdot) \|_{L_1}\geq S \}$ is well-defined and $\tau_m(S)\to\infty$ in probability as $S\to\infty$. Then, we obtain \eqref{non_explosive_limit} by considering the local solution $u_m$ on $[0,\tau_m(S))$; see Lemma \ref{Non_explosion_small_lambda_STWN}
\end{remark}

\begin{lemma} \label{Non_explosion_small_lambda_STWN}
Suppose all the conditions of Theorem \ref{theorem_burgers_eq_STWN_super_linear} hold, and $u_m$ is the solution to equation \eqref{cut_off_equation_small_lambda_STWN} introduced in Lemma \ref{cut_off_lemma_small_lambda_STWN}. Then, we have
\begin{equation} 
\label{stopping_time_blow_up_STWN_2}
\lim_{R\to\infty}\sup_mP\left( \left\{\omega\in\Omega : \sup_{t\leq T,x\in\bR} |u_m(t,x)|\geq R \right\} \right) = 0.
\end{equation}
\end{lemma}
\begin{proof}
Let $T\in(0,\infty)$. For $m,S\in\bN$, set
\begin{equation*}
\tau_m(S):=\inf\{t\geq0: \| u_m(t,\cdot) \|_{L_1}\geq S\}\wedge T.
\end{equation*}
With the help of Lemma \ref{L_1_bound_small_lambda_STWN_lemma}, the random time $\tau_m(S)$ is a well-defined stopping time. 

Let $t\in(0,T)$. Then, Theorem \ref{theorem_nonlinear_case}, Remark \ref{Kernel}, H\"older inequality, and Mink\"owski's inequality imply
\begin{equation}
\label{non_explosive_computation_1}
\begin{aligned}
&\| u_m \|^p_{\cH_{p}^{1/2-\kappa}(\tau_m(S)\wedge t)} - N\| u_0 \|^p_{U_p^{1/2-\kappa}}\\
&\quad\leq N\left\| \bar b \left( u_m^{1+\lambda } \right)_{x} \right\|_{\bH_{p}^{-3/2-\kappa}(\tau_m(S)\wedge t)}^p + N\left\| \mu u_m^{1+\lambda_0}\eta \right\|_{\bH_{p}^{-1/2-\kappa}(\tau_m(S)\wedge t,\ell_2)}^p \\
&\quad\leq N\left\| u_m^{1+\lambda } \right\|_{\bH_{p}^{-1/2-\kappa}(\tau_m(S)\wedge t)}^p + N\left\| \mu u_m^{1+\lambda_0}\eta \right\|_{\bH_{p}^{-1/2-\kappa}(\tau_m(S)\wedge t,\ell_2)}^p \\
&\quad\leq  N\bE\int_0^{\tau_m(S)\wedge t}\int_{\bR} \left( \int_{\bR} R_{1/2+\kappa}(x-y)|u_m(s,y)|^{1+\lambda } dy \right)^p dxds \\
&\quad\quad + N\bE\int_0^{\tau_m(S)\wedge t}\int_{\bR}\left(  \int_{\bR} \left|R_{1/2+\kappa}(x-y)\right|^2|u_m(s,y)|^{2+2\lambda_0} dy  \right)^{p/2} dxds.
\end{aligned}
\end{equation}
Observe that 
\begin{equation}
\label{non_explosive_computation_2}
\begin{aligned}
&\int_{\bR} \left( \int_{\bR} R_{1/2+\kappa}(x-y)|u_m(s,y)|^{1+\lambda } dy \right)^p dx \\
&\quad \leq \|u_m(s,\cdot)\|_{L_1}^{p\lambda }\int_{\bR} \left( \int_{\bR} |R_{1/2+\kappa}(x-y)|^{\frac{1}{1-\lambda }}|u_m(s,y)|^{\frac{1}{1-\lambda }} dy \right)^{p(1-\lambda )} dx \\
&\quad \leq S^{p\lambda }\| R_{1/2+\kappa} \|_{L_{\frac{1}{1-\lambda }}(\bR)}^p\int_{\bR}|u_m(s,x)|^p dx
\end{aligned}
\end{equation}
and
\begin{equation}
\label{non_explosive_computation_3}
\begin{aligned}
&\int_{\bR}\left(  \int_{\bR} \left|R_{1/2+\kappa}(x-y)\right|^2|u_m(s,y)|^{2+2\lambda_0} dy  \right)^{p/2} dx \\
&\quad \leq \|u_m(s,\cdot)\|_{L_1}^{p\lambda_0}\int_{\bR} \left(  \int_{\bR} |R_{1/2+\kappa}(x-y)|^{\frac{2}{1-2\lambda_0}}|u_m(s,y)|^{\frac{2}{1-2\lambda_0}} dy  \right)^{p(1/2-\lambda_0)} dx\\
&\quad \leq S^{p\lambda_0}\| R_{1/2+\kappa} \|_{L_{\frac{2}{1-2\lambda_0}}(\bR)}^p \int_{\bR}|u_m(s,x)|^p dx
\end{aligned}
\end{equation}
for all $s\in(0,\tau_m(S)\wedge t)$ almost surely. Therefore, by applying \eqref{non_explosive_computation_2} and \eqref{non_explosive_computation_3} to \eqref{non_explosive_computation_1}, we have
\begin{equation*}
\begin{aligned}
&\| u_m \|^p_{\cH_{p}^{1/2-\kappa}(\tau_m(S)\wedge t)} - N\| u_0 \|^p_{U_p^{1/2-\kappa}}\\
&\quad\leq N\left( S^{p\lambda }\| R_{1/2+\kappa} \|_{L_{\frac{1}{1-\lambda }}(\bR)}^p + S^{p\lambda_0}\| R_{1/2+\kappa} \|_{L_{\frac{2}{1-2\lambda_0}}(\bR)}^p \right)\|u_m\|^p_{\bL_p(\tau_m(S)\wedge t)},
\end{aligned}
\end{equation*}
where $N = N(p,\kappa,K,T)$. Note that $\| R_{1/2+\kappa} \|_{L_{\frac{1}{1-\lambda }}(\bR)} + \| R_{1/2+\kappa} \|_{L_{\frac{2}{1-2\lambda_0}}(\bR)} < \infty$ since $\kappa > (\lambda  - 1/2)\vee \lambda_0$. Thus,
\begin{equation*}
\| u_m \|_{\cH_p^{1/2-\kappa}(\tau_m(S)\wedge t)}^p \leq N_1\| u_0 \|^p_{U_p^{1-\kappa}} + N_2\| u_m \|_{\bL_p(\tau_m(S)\wedge t)}^p,
\end{equation*}
where $N_1 = N_1(p,\kappa,K,T)$ and  $N_2 = N_2(S,\lambda ,\lambda_0,p,\kappa,K,T)$. By Corollary \ref{embedding_corollary} and Theorem \ref{embedding} \eqref{gronwall_type_ineq}, we have
\begin{equation*}
\bE\sup_{t\leq \tau_m(S)\wedge T,x\in\bR}|u_m(s,x)|^p \leq N\| u_m \|^p_{\cH_{p}^{1-\kappa}(\tau_m(S)\wedge T)} \leq N  \|u_0\|_{U_p^{1-\kappa}}^p,
\end{equation*}
where $N = N(S,\lambda ,\lambda_0,p,\kappa,K,T)$. By the way, by \eqref{L_1_bound_small_lambda_STWN} and Chebyshev's inequality, we have
\begin{equation*}
\begin{aligned}
P\left( \sup_{t\leq T}\| u_m(t,\cdot) \|_{L_1}\geq S \right) \leq \frac{1}{\sqrt S}\bE\sup_{t\leq T}\| u_m(t,\cdot) \|_{L_1}^{1/2}  \leq \frac{N}{\sqrt S},
\end{aligned}
\end{equation*}
where $N = N(K,T)$.
Therefore, by Chebyshev's inequality twice, we have
\begin{equation*}
\begin{aligned}
&P\left(\sup_{t\leq T,x\in\bR}|u_m(t,x)| > R\right) \\
&\quad\leq P\left(\sup_{t\leq \tau_m(S)\wedge T,x\in\bR}|u_m(t,x)| > R\right) + P(\tau_m(S)<T)  \\
\end{aligned}
\end{equation*}
\begin{equation*}
\begin{aligned}
&\quad\leq P\left(\sup_{t\leq \tau_m(S) \wedge T,x\in\bR}|u_m(t,x)| > R\right) + P\left( \sup_{t\leq T}\| u_m(t,\cdot) \|_{L_1}\geq S \right)\\
&\quad\leq \frac{N_1}{R^p} + \frac{N_2}{\sqrt S},
\end{aligned}
\end{equation*}
where $N_1 = N_1(S,\lambda ,\lambda_0,p,\kappa,K,T)$ and $N_2 = N_2(K,T)$. By taking the supremum with respect to $m$,  letting $R\to\infty$, $S\to\infty$ in order, we get \eqref{stopping_time_blow_up_STWN_2}. The lemma is proved.
\qed
\end{proof}

Now we introduce the proof of Theorem \ref{theorem_burgers_eq_STWN_super_linear}. The motivation of the proof follows from \cite[Section 8.4]{kry1999analytic} and Proof of Theorem 2.11 in \cite{han2019boundary}.

\vspace{2mm}

\noindent\textbf{Proof of Theorem \ref{theorem_burgers_eq_STWN_super_linear}}
{\it Step 1. (Uniqueness).}  
Follow the \textit{Step 1} of \textbf{Proof of Theorem \ref{theorem_burgers_eq_large_lambda_STWN}}. The only difference is $\kappa\in(0,(\lambda - 1/2)\vee \lambda_0,1/2)$ instead of $\kappa\in(0,1/2)$.

{\it Step 2. (Existence).} 
Let $\kappa\in(0,(\lambda - 1/2)\vee \lambda_0,1/2)$ and $T\in(0,\infty)$. By Lemma \ref{cut_off_lemma_small_lambda_STWN}, there exists nonnegative $u_m\in\cH_{p}^{1/2-\kappa}(T)$ satisfying equation \eqref{burger's_eq_space_time_white_noise_super_linear}. By Corollary \ref{embedding_corollary}, we have $u_m\in C([0,T];C(\bR))$ (a.s.). Thus, we can define $\tau_m^R$ as in \eqref{stopping_time_taumm}.
Also, as in Remark \ref{def_u}, we have $\tau_R^R \leq \tau_m^m$. 
Therefore, Lemma \ref{Non_explosion_small_lambda_STWN} implies 
\begin{equation*}
\begin{aligned}
\limsup_{m\to\infty}  P\left( \tau_m^m < T \right) &= \limsup_{m\to\infty} P\left(\sup_{t\leq T,x\in\bR}|u_m(t,x)|\geq m\right) \\
&\leq \limsup_{m\to\infty}\sup_{n} P\left(\sup_{t\leq T,x\in\bR}|u_n(t,x)|\geq m\right)\to 0,
\end{aligned}
\end{equation*}
which implies $\tau_m^m\to \infty$ in probability. Since $\tau_m^m$ is increasing,  we conclude that $\tau_m^m\uparrow \infty$ (a.s.). Define $\tau_m := \tau_m^m\wedge m$ and
\begin{equation*}
u(t,x):=u_m(t,x)\quad \text{for}~ t\in[0,\tau_m].
\end{equation*}
Note that $u$ satisfies \eqref{burger's_eq_space_time_white_noise_super_linear} for all $t \leq \tau_m$, because $|u(t)| = |u_m(t)|\leq m$ for $t\leq\tau_m$. Since $u = u_m \in \cH_p^{1-\kappa}(\tau_m)$ for any $m$, we have $u\in\cH_{p,loc}^{1-\kappa}$. 

{\it Step 3. (H\"older regularity).}  Let $T<\infty$. Since $u\in \cH_p^{1/2-\kappa}(T)$, by employing Corollary \ref{embedding_corollary}, we have \eqref{holder_regularity_large_lambda_STWN_2}. The theorem is proved.

\qed

\noindent\textbf{Proof of Theorem \ref{uniqueness_in_p_2}} Follow the proof of \textbf{Proof of Theorem \ref{uniqueness_in_p_1}}. It should be mentioned that we assume $\kappa\in(0,(\lambda - 1/2)\vee \lambda_0,1/2)$.
\qed


\bibliographystyle{plain}

\end{document}